\documentclass[12pt]{article}
\usepackage{subcaption}
\usepackage[font=scriptsize]{caption}

\usepackage{latexsym,amsmath,amsthm,amssymb,amsfonts,amscd,multirow}
\usepackage{epsfig,verbatim,epstopdf,graphics}
\usepackage{color}
\usepackage{array}
\usepackage[ruled,linesnumbered]{algorithm2e}
\usepackage{mathtools} %

\newtheorem{Example}{Example}[section]
\newtheorem{Thm}{Theorem}[section]
\newtheorem{Lem}[Thm]{Lemma}
\newtheorem{rem}{Remark}[section]
\newtheorem{definition}{Definition}[section]

\DeclareMathOperator{\sech}{sech}

\newcommand{\ot}{\frac{1}{2}}
\newcommand{\oh}{\frac{1}{h}}
\newcommand{\toh}{\frac{2}{h}}

\newcommand{\intj}{\int_{I_j}}

\newcommand{\ub}{\bar u_h}	
\newcommand{\uhat}{\hat u_h}
\newcommand{\uh}{u_h}
\newcommand{\vh}{v_h}
\newcommand{\uxt}{\widetilde{(u_h)_x}}
\newcommand{\jp}{{j+\frac{1}{2}}}
\newcommand{\jm}{{j-\frac{1}{2}}}
\newcommand{\uxa}{\{ (u_h)_x \}}
\newcommand{\ua}{\{ u_h \}}

\newcommand{\sumj}{\sum_j}

\newcommand{\ao}{\alpha_1}
\newcommand{\at}{\alpha_2}
\newcommand{\bo}{\beta_1}
\newcommand{\bt}{\beta_2}

\newcommand{\s}{\mathcal{S}}
\newcommand{\w}{\mathcal{W}}

\newcommand{\lo}{\lambda_1}
\newcommand{\lt}{\lambda_2}
\newcommand{\inva}{A^{-1}}
\newcommand{\lambdao}{\frac{\lo^j}{1-\lo^N}}
\newcommand{\lambdat}{\frac{\lt^j}{1-\lt^N}}

\newcommand{\doj}{d_1^j}
\newcommand{\dtj}{d_2^j}

\newcommand{\sumn}{\sum_{n=-\infty}^{\infty}}

\newcommand{\pkpos}{P_{k}(1)}
\newcommand{\pkneg}{P_{k}(-1)}
\newcommand{\pkmopos}{P_{k-1}(1)}
\newcommand{\pkmoneg}{P_{k-1}(-1)}
\newcommand{\dpkpos}{P^{'}_{k}(1)}
\newcommand{\dpkneg}{P^{'}_{k}(-1)}
\newcommand{\dpkmopos}{P^{'}_{k-1}(1)}
\newcommand{\dpkmoneg}{P^{'}_{k-1}(-1)}

\newcommand{\sch}{Schr\"odinger }
\newcommand{\beq}{\begin{equation}}
\newcommand{\eeq}{\end{equation}}
\newcommand{\beqa}{\begin{eqnarray}}
\newcommand{\eeqa}{\end{eqnarray}}

\newcommand{\ppm}{P^{1}_h}
\newcommand{\pt}{P^2_h}
\newcommand{\pst}{P^\star_h}

\newcommand{\norm}[1]{\left\lVert#1\right\rVert}
\newcommand{\abs}[1]{\left | #1 \right |}
\newcommand{\la}{\Lambda}
\newcommand{\g}{\Gamma}

\DeclarePairedDelimiter{\floor}{\lfloor}{\rfloor}

\newcolumntype{H}{>{\setbox0=\hbox\bgroup}c<{\egroup}@{}}
\title
{An Ultra-Weak Discontinuous Galerkin Method for \sch Equation in One Dimension}

\author{Anqi Chen
\thanks{Department of Mathematics, Michigan State University,
East Lansing, MI 48824 U.S.A.
{\tt chenaq3@msu.edu}.}%
 \and
 Fengyan Li
\thanks{Department of Mathematical Sciences, Rensselaer Polytechnic Institute, Troy, NY 12180 U.S.A.
 {\tt  lif@rpi.edu}. Research is supported by NSF grant  DMS-1719942.}
\and
 Yingda Cheng
\thanks{Department of Mathematics, Department of  Computational Mathematics, Science and Engineering, Michigan State University,
East Lansing, MI 48824 U.S.A.
 {\tt ycheng@msu.edu}. Research is supported by NSF grants  DMS-1453661 and DMS-1720023.}
}

\date{\today}

\begin{document}
\maketitle

\begin{abstract}
In this paper, we develop an ultra-weak discontinuous Galerkin (DG) method to solve the one-dimensional nonlinear \sch equation. Stability conditions and error estimates are derived for the scheme with a general class of numerical fluxes. The error estimates are based on detailed analysis of the projection operator  associated with each individual flux choice. Depending on the parameters, we find out that in some cases, the projection can be defined  element-wise, facilitating analysis. In most cases, the projection is global, and its analysis depends on the resulting $2\times2$ block-circulant matrix structures.  For a large class of parameter choices, optimal \emph{a priori} $L^2$ error estimates can be obtained. 
Numerical examples are provided verifying theoretical results.
\end{abstract}


\section{Introduction}
In this paper, we develop and analyze a discontinuous Galerkin (DG) method for  one-dimensional nonlinear \sch (NLS) equation:
\beq
\label{eqn:nls1}
iu_t+u_{xx}+f(|u|^2)u=0,
\eeq
where $f(u)$ is a  nonlinear real function and $u$ is a complex function.
The \sch equation is the fundamental equation in quantum mechanics, reaching out to many applications
in fluid dynamics, nonlinear optics and plasma physics. It is also called \sch wave equation as it can 
describe how the wave functions of a physical system evolve over time.  Many numerical methods have been applied to solve NLS equations \cite{chang1999difference,dag1999quadratic,karakashian1998space,karakashian1999space,pathria1990pseudo,sheng2001solving,taha1984analytical}. In \cite{chang1999difference,taha1984analytical}, several important finite difference schemes are implemented, analyzed and compared. In \cite{pathria1990pseudo}, the author introduced a pseudo-spectral method for general NLS equations. Many finite element methods have been tested, such as quadratic B-spline for NLS in \cite{dag1999quadratic,sheng2001solving} and space-time DG method for nonlinear (cubic) \sch equation in \cite{karakashian1998space,karakashian1999space}.  In this paper, we  focus on the DG methods,  which is a class of finite element methods using completely discontinuous piecewise function space for test functions and numerical solution, to solve the \sch equation. The first DG method was introduced by Reed and Hill in \cite{reed1973triangular}. A major development of DG methods is the Runge-Kutta DG (RKDG) framework introduced for solving hyperbolic conservation laws containing only first order spatial derivatives in a series of papers \cite{cockburn1998ldg1,cockburn1989ldg2,cockburn1989ldg3,cockburn1990ldg4,cockburn1998ldg5}. Because of the completely discontinuous basis, DG methods have several attractive properties. It can be used on many types of meshes, even those with hanging nodes. The methods have $h$-$p$ adaptivity and very high parallel efficiency.
 
Various types of DG schemes for discretizing the second order spatial derivatives have been used to compute \eqref{eqn:nls1}. One group of such methods is the so-called local DG
(LDG) method invented in \cite{cockburn1998ldg1} for convection-diffusion equations. The algorithm is based on introducing auxiliary variables and reformulating the equation into its first order form. In \cite{xu2005local}, a LDG method using alternating fluxes is developed with $L^2$ stability and proved $(k+\ot)$-th order of accuracy. Later in \cite{MR2888305}, Xu and Shu  proved optimal accuracy for both the solution and the auxiliary variables in the LDG method for high order wave equations based on refined energy estimates. 
In \cite{XingETDLDGsch}, the authors presented a LDG method with  exponential time differencing Runge-Kutta scheme and investigated the energy conservation performance of the scheme. Another group of method involves treating the second order spatial derivative directly in the weak formulations, such as IPDG method \cite{wheeler1978elliptic,douglas1976interior} and NIPG method \cite{riviere1999improved,riviere2001priori}. Those schemes enforce a penalty jump term in the weak formulation, and they have been extensively applied  to acoustic and elastic wave propagations \cite{grote2006discontinuous, ainsworth2006dispersive, riviere2003discontinuous}. %
 As for \sch equations, the direct DG (DDG) method was   applied to \sch equation in \cite{luddgsch} and achieved energy conservation and optimal accuracy. %
 Among all those various formulations, the work in this paper focus on the ultra-weak DG methods,  which can be traced backed to \cite{cessenat1998application}, and  refer to those DG methods \cite{shu2016discontinuous} that rely on repeatedly applying integration by parts  so all the spatial derivatives are shifted from the solution to the test function in the weak formulations.   In \cite{MR2373176}, Cheng and Shu developed   ultra-weak DG methods for general time dependent problems with higher order spatial derivatives. In \cite{XingKdvDG}, Bona \emph{et. al.} proposed an ultra-weak DG scheme for generalized KdV equation and performed error estimates.

The focus of this paper is the investigation of   a most general form of  the numerical flux functions that ensures stability along with our ultra-weak formulation. The fluxes under consideration include the alternating fluxes, and also the fluxes considered in  \cite{luddgsch}, and therefore allows for flexibility for the design of the schemes.  It is widely known that the choice of flux  can have significant impact on the convergence order of the scheme as evidenced in DG methods for linear first-order transport equations,   two-way wave equations  \cite{cheng20152}, and the KdV equations \cite{MR2373176, XingKdvDG} and many others. The main contribution of the work is a systematic study of error estimates based on the flux parameters.  To this end, we define and analyze   projection  operator  associated with each specific parameter choice.  We assume the dependence of parameters on the mesh size can be freely enforced, therefore many cases shall follow.  We find out that under certain conditions, the projections are ``local", meaning that they can be defined element-wise. In the most general setting, the projections are global, and detailed analysis based on block-circulant matrices are necessary. This type of analysis has been done in  \cite{XingKdvDG, MengDGerror_est} for circulant matrices  and in \cite{liu2015optimal} for block-circulant matrices, but our case is more involved due to the $2 \times 2$ block-circulant structure, for which several cases need to be distinguished based on the eigenvalues of the block matrices, and some requires  tools from Fourier analysis. Our analysis reveals that under a large class of parameter choices, our method is optimally convergent in $L^2$ norm, which is verified by extensive numerical tests for both the projection operators and the numerical schemes for \eqref{eqn:nls1}.

The remainder of this paper is organized as follows. In Section \ref{sec:scheme}, we introduce an ultra-weak DG method with general flux definitions for one-dimensional 
nonlinear \sch equations and study its stability properties. The main body of the paper, the error estimates, is contained in Section \ref{sec:error}.
We introduce a new projection operator and analyze its properties in Section \ref{sec:projection}, which is  later used in Section \ref{sec:l2estimate} to obtain the convergence results of the schemes. Numerical validations   are provided in Section \ref{sec:numerical}. Conclusions are made in Section \ref{sec:conclusion}. Some technical details, including proof of most lemmas are collected in the Appendix.

\section{A DG Method for One-Dimensional \sch Equations}
\label{sec:scheme}
In this section, we formulate and discuss stability results of a DG scheme for  one-dimensional NLS equation \eqref{eqn:nls1}
on interval $I=[a,b]$ with initial condition
$
u(x,0)=u_0(x)
$
and periodic boundary conditions. Here $f(u)$ is a given real function. Our method can be defined for general boundary conditions, but the error analysis will require slightly different tools, and therefore we only consider periodic boundary conditions in this paper.

To facilitate the discussion, first we introduce some notations and definitions.
For a  1-D interval $I=[a,b]$, the usual DG meshes are defined as:
\[
a=x_{\ot} < x_{\frac{3}{2}} < \cdots < x_{N+\ot} =b,
\]
\[
I_j=(x_{\jm},x_{\jp}), \quad x_j=\ot(x_{\jm}+x_{\jp}),
\]
and
\[
h_j=x_{\jp}-x_{\jm}, \quad h=\max_j h_j,
\]
with mesh regularity requirement $\frac{h}{\min h_j} < \sigma$, $\sigma$ is fixed during mesh refinement.

The approximation   space is defined as:
\[
V_h^k=\{ v_h: v_h|_{I_j} \in P^k(I_j),\, j=1, \, \cdots, N \},
\]
meaning $v_h$ is a polynomial of degree up to $k$ on each cell $I_j$. For a function
$v_h \in V_h^k$, we use $(v_h)^-_{\jm}$ and $(v_h)^+_{\jm}$ to refer to the value of $v_h$ at $x_{\jm}$
from the left cell $I_{j-1}$ and the right cell $I_j$ respectively. The jump and average are defined as $[v_h]=v_h^+-v_h^-$ and
 $\{v_h\}=\ot(v_h^++ v_h^-)$ at cell interfaces.

In this paper, we consider a DG scheme   motivated by \cite{MR2373176} and based on integration by parts twice, or the so-called ultra-weak formulation. In particular, we
look for the unique function $u_h=u_h(t) \in V_h^k, \, t\in (0,T]$, such that
\begin{align}
\label{eqn:scheme}
i\intj (u_h)_t \vh dx + \intj u_h (\vh)_{xx} dx - \uhat (\vh)^-_x |_\jp + \uhat (\vh)^+_x |_\jm  \nonumber \\ 
+ \uxt \vh^- |_\jp -\uxt \vh^+ |_\jm + \intj f(|u_h|^2)u_hv_h dx = 0
\end{align}
holds for all $\vh \in V_h^k$ and all $j=1,\, \cdots , N$. Here, we require $k \ge 1$, because $k=0$ yields a inconsistent scheme. Notice that \eqref{eqn:scheme} can be written equivalently in a weak formulation by performing another integration by parts back as: 
\begin{align}
\label{eqn:scheme2}
i\intj (u_h)_t \vh dx - \intj (u_h)_x (\vh)_{x} dx +(u_h^- - \uhat) (\vh)^-_x |_\jp + (\uhat-u_h^+) (\vh)^+_x |_\jm  \nonumber \\ 
+ \uxt \vh^- |_\jp -\uxt \vh^+ |_\jm + \intj f(|u_h|^2)u_hv_h dx = 0
\end{align}

The ``hat" and``tilde" terms are the numerical fluxes we pick for $u$ and $u_x$ at cell boundaries, which 
are single valued functions defined as:
\beq
\label{eqn:flux}
\uxt=\uxa+\ao [(u_h)_x] +\bo[u_h],  \quad \uhat=\ua +\at [u_h] +\bt[(u_h)_x], \quad \ao,\at \in \mathbb{C}, \, \bo,\bt \in \mathbb{C},
\eeq
where $\ao, \at, \bo, \bt$ are prescribed parameters. They may depend on the mesh parameter $h.$ Commonly used fluxes such as the central flux (by setting $\ao=\at=\bo=\bt=0$) and alternating fluxes (by setting  $\ao=-\at=\pm \frac{1}{2}, \bo=\bt=0$) belong to this flux family.
The direct DG scheme considered in \cite{luddgsch} is a special case of our method when $\ao = -\at, \bo = \frac{c}{h},\bt=0,  c>0, \ao \in \mathbb{R}$.  The IPDG method can also be casted in this framework as $\ao = \at = \bt = 0, \bo = \frac{c}{h}, c>0.$ 

Using periodic boundary condition, we can sum up on $j$ for the numerical scheme \eqref{eqn:scheme} and reduce it into the following short-hand notation
\beq
\label{eqn:ah}
 a_{\ao, \at, \bo, \bt}(u_h,v_h)-i \int_I f(|u_h|^2)u_hv_h dx=0,
\eeq
where
\begin{equation*}
 a_{\ao, \at, \bo, \bt}(u_h,v_h) =\int_I (u_h)_t \vh dx -i \int_I u_h (\vh)_{xx} dx -i\sumj (\uhat [(\vh)_x]-\uxt[\vh])|_\jp.
\end{equation*}

The following theorem contains the results on semi-discrete $L^2$ stability.

\begin{Thm}
\label{thm:stability}
\textup{(Stability)} The solution of semi-discrete DG scheme \eqref{eqn:scheme} using numerical fluxes \eqref{eqn:flux} satisfies
$L^2$ stability condition
\[
\frac{d}{dt}\int_I |u_h|^2 dx \leq0,
\]
if 
\beq
\label{eqn:par}
\textup{Im}\bt \geq 0, \, \textup{Im}\bo \leq 0,\, |\ao+\overline{\alpha_2}|^2 \le -4 \textup{Im}\bo \textup{Im}\bt.
\eeq
In particular, when all parameters $\ao, \at, \bo, \bt$ are restricted to be real, this condition amounts to
\beq
\label{eqn:par2}
\ao+\at=0
\eeq
without any requirement on $\bo, \bt$. 
\end{Thm}
\begin{proof} %
From integration by parts, we have, for $\forall v_h \in V_h^k$
\begin{align}
a_{\ao, \at, \bo, \bt}(u_h,v_h)=
 \int_I (u_h)_t \vh dx +i\int_I (u_h)_x (\vh)_{x} dx +i\sumj ([u_h (v_h)_x]-\uhat [(\vh)_x]+\uxt[\vh])|_\jp. \notag
\end{align}

Taking $\vh=\bar u_h$ in \eqref{eqn:ah} and compute its conjugate as well, we get
\begin{align}
0&=i \int_I f(|u_h|^2)|u_h|^2 dx + \overline{i \int_I f(|u_h|^2)|u_h|^2 dx } \notag \\
&=a_{\ao, \at, \bo, \bt}(u_h,\bar u_h) + \overline{a_{\ao, \at, \bo, \bt}(u_h,\bar u_h)} \notag \\
&=\frac{d}{dt} \int_I |u_h|^2 dx - 2\textrm{Im} \sumj ([\uh(\bar u_h)_x] - \uhat [(\bar u_h)_x]+ \uxt [\ub])|_\jp. \label{eqn:stab2}
\end{align}

Define
\begin{align*}
A(\uh,\bar u_h) = &\sumj ([\uh(\bar u_h)_x] - \uhat [(\bar u_h)_x]+ \uxt [\ub])|_\jp \\
	          = &\sumj (\{\uh\}[(\bar u_h)_x]+[\uh]\{(\bar u_h)_x\}-\{\uh\}[(\bar u_h)_x]-\at[\uh][(\bar u_h)_x]-\bt [(u_h)_x][(\bar u_h)_x]\\
	          &+\{{(u_h)_x}\}[\bar u_h] +\ao[(u_h)_x][\bar u_h]+\bo[\uh][\bar u_h])|_\jp \\
		 = & \sumj \big(2\textrm{Re}([\uh]\{(\bar u_h)_x\})-\bt|[(u_h)_x]|^2+\bo|[\uh]|^2+\ao[(u_h)_x][\bar u_h]-\at[\uh][(\bar u_h)_x] \big)|_\jp.
\end{align*}
	Therefore, Im$A(\uh,\bar u_h)=\sumj (-\textrm{Im}\bt|[(u_h)_x]|^2+\textrm{Im}\bo|[\uh]|^2+\textrm{Im}\{(\ao+ \overline{\at})[\bar u_h][(u_h)_x]\})|_\jp$. Plug it back into \eqref{eqn:stab2}: 

\beq
\label{eqn:stab3}
\frac{d}{dt} \int_I |u_h|^2 dx+
\sumj 2\textrm{Im}\bt|[(u_h)_x]|^2-2\textrm{Im}\bo|[\uh]|^2-2\textrm{Im}\{(\ao+  \overline{\at})[\bar u_h][(u_h)_x]\}|_\jp=0.
\eeq

If the stability condition \eqref{eqn:par} is satisfied, we have
$$
\frac{d}{dt} \int_I |u_h|^2 dx \leq   0.
$$

If   all parameters are real and \eqref{eqn:par2} is satisfied, then \eqref{eqn:stab3} further yields:
$$
\frac{d}{dt} \int_I |u_h|^2 dx = 0,
$$
which implies energy conservation.
\end{proof}

\begin{rem}
For simplicity of the discussion, in the next section, we will only consider  real parameters, i.e. when $\ao, \at, \bo, \bt$ are real and $\ao+\at=0$. 
\end{rem}

\section{Error Estimates}
\label{sec:error}

In this  section, we will derive error estimates of the DG scheme \eqref{eqn:scheme} for the model NLS equation \eqref{eqn:nls1}.
 As mentioned before, we consider $L^2$ stable real parameter choices, which means the numerical fluxes are defined by three parameters as,
\beq
\label{eqn:flux2}
\uxt=\uxa+\ao [(u_h)_x] +\bo[u_h],  \quad \uhat=\ua -\ao [u_h] +\bt[(u_h)_x], \quad \ao, \bo,\bt \in \mathbb{R}.
\eeq

We will focus on the impact of the choice of the parameters $\ao, \bo, \bt$ on the accuracy of the scheme.  
We proceed as follows: first, we   define and discuss the properties of projection operator $\pst$ in Section \ref{sec:projection}. Then, we use the projection error estimates to obtain convergence result for DG scheme in Section \ref{sec:l2estimate}.

\subsection{Projection Operator}
\label{sec:projection}
In this subsection, we perform detailed studies of a projection operator defined as follows.

\begin{definition}
For our DG scheme with flux choice \eqref{eqn:flux2}, we define the associated projection operator $\pst$ for any periodic function $u \in W^{1,\infty}(I) $  
to be the unique polynomial $\pst u \in V_h^k$ (when $k \geq 1$) satisfying  
\begin{subequations}
\label{eqn:pst}
\begin{align} \label{eqn:psts1}
\intj \pst u  \, \vh dx &=\intj u  \,\vh dx \quad & &\forall \vh \in P^{k-2}(I_j),\\
\label{eqn:psts2}
\widehat{\pst u}  =\{\pst u\} -\ao [\pst u] +\bt[(\pst u)_x] &= u  & &\textrm{at} \quad x_{\jp}, \\
\label{eqn:psts3}
\widetilde{\pst u_x}=\{(\pst u)_x\}+\ao [(\pst u)_x] +\bo[\pst u] & =u_x & &\textrm{at} \quad x_{\jp},
\end{align}
\end{subequations}
for all $j$. When
$k=1$,  only conditions \eqref{eqn:psts2}-\eqref{eqn:psts3}  are needed.
\end{definition}

 This definition is to ensure $\widehat{u - \pst u} =0$ and $\widetilde{u_x - \pst u_x} = 0$, which will be used in error estimates for the scheme.  In the following, we  analyze the projection when the parameter choice reduces it to a local projection in Section \ref{sec:localp}, and then we consider the more general global projection in Section \ref{sec:globalp}.

\subsubsection{Local projection results}
\label{sec:localp}
In general, the projection $\pst$ is globally defined, and its existence, uniqueness and approximation properties are quite complicated mathematically. However, with some special parameter choices, $\pst$ can be reduced to a local projection, meaning that it can be solved element-wise, and hence the analysis can be greatly simplified. 

 For example, with the alternating fluxes $\ao=\pm \frac{1}{2}, \bo=\bt=0,$ $\pst$ can be reduced to $\ppm$ and $\pt$ defined below.
$\pst=\ppm$ for parameter choice $\ao= \frac{1}{2}, \bo=\bt=0$ is formulated as:  for each cell $I_j$, we find the unique polynomial of degree $k$,  $\ppm u$, satisfying

\begin{subequations}
\label{eqn:ppm}
\begin{align} 
\label{eqn:ppm1} \intj \ppm u \, \vh dx &=\intj u \,\vh dx \quad & &\forall \vh \in P^{k-2}(I_j), \\
\label{eqn:ppm2} (\ppm u)^- & = u  & &\textrm{at} \, x_{\jp}, \\
\label{eqn:ppm3} (\ppm u)_x^+& = u_x & &\textrm{at} \, x_{\jm}.
\end{align}
\end{subequations}
When $k=1$, only conditions \eqref{eqn:ppm2}-\eqref{eqn:ppm3}  are needed.

Similarly, we can define $\pst=\pt$  for parameter choice $\ao= -\frac{1}{2}, \bo=\bt=0$ as:  for each cell $I_j$, we find the unique polynomial of degree $k$,  $\pt u$, satisfying
\begin{subequations}
\label{eqn:pt}
\begin{align}
\label{eqn:pt1} \intj \pt u \, \vh dx &=\intj u \,\vh dx \quad & &\forall \vh \in P^{k-2}(I_j), \\
\label{eqn:pt2} (\pt u)^+ & = u  & &\textrm{at} \, x_{\jm}, \\
\label{eqn:pt3} (\pt u)_x^-& = u_x & &\textrm{at} \, x_{\jp}.
\end{align}
\end{subequations}
When $k=1$, only conditions \eqref{eqn:pt2}-\eqref{eqn:pt3}  are needed.

Similar local projections have been introduced and considered in \cite{MR2373176}. It is obvious that $\ppm u, \pt u$  can be solved element-wise, and their existence, uniqueness are straightforward.  From a standard scaling argument by Bramble-Hilbert lemma in \cite{ciarletFEM},  $\ppm$ and $\pt$ have the following error estimates: let $u\in W^{k+1,p}(I_j) (p=2,\infty)$, then
\beq
\label{eqn:ppmerr}
\begin{aligned}
\|u-P^{\nu}_h u\|_{L^p(I_j)} \leq Ch^{k+1}_j |u|_{W^{k+1,p}(I_j)}, \quad p=2,\infty, \ \nu=1,2, \\
\|u_x-P^{\nu}_h u_x\|_{L^p(I_j)} \leq Ch^{k}_j |u|_{W^{k+1,p}(I_j)}, \quad p=2,\infty, \ \nu=1,2,
\end{aligned}
\eeq
where here and below, $C$ is a generic constant that is independent of the mesh size $h_j$, the parameters $\ao, \bo, \bt$ and the function $u$, but may take different value in each occurrence.

Naturally, the next question is that if there are other parameter choices such that $\pst$ can be reduced to a local projection. The following lemma addresses this issue.

\begin{Lem}[The condition for reduction to a local projection]
 If $\ao^2+\bo\bt=\frac{1}{4}$,  $\pst$ is a local
projection.
\end{Lem}
\begin{proof}

We can write   \eqref{eqn:psts2}-\eqref{eqn:psts3} as
\begin{align}
u = & (\ot+\ao)(\pst u)^- -\bt (\pst u)_x^- + (\ot-\ao)(\pst u)^+ +\bt (\pst u)_x^+  & &\textrm{at}  \, x_{\jp}, \, \forall j,
\label{eqn:ucomp}\\
u_x = & -\bo (\pst u)^-+(\ot-\ao)(\pst u)_x^- +\bo (\pst u)^+ +(\ot+\ao)(\pst u)_x^+ & &\textrm{at}  \, x_{\jp}, \, \forall j.
\label{eqn:uxcomp}
\end{align}

By simple algebra, if $\ao^2+\bo\bt=\frac{1}{4}$, we obtain:
\begin{itemize}
\item if $\bo \neq 0$, then at $x_{\jp}$ for all $j,$ we have
\begin{equation}
\label{eqn:cancel1}
\begin{aligned}
u+\frac{\ot+\ao}{\bo}u_x & = (\pst u)^+ +(\bt+\frac{(\ot+\ao)^2}{\bo})(\pst u)_x^+ = (\pst u)^+ +\frac{\ot+\ao}{\bo}(\pst u)_x^+ ,  \\
u-\frac{\ot-\ao}{\bo}u_x & = (\pst u)^- -(\bt+\frac{(\ot-\ao)^2}{\bo})(\pst u)_x^- = (\pst u)^- -\frac{\ot-\ao}{\bo}(\pst u)_x^-,
\end{aligned}
\end{equation}
meaning that  $\pst$ can be defined element-wise on cell $I_j$ as:
\beq
\label{eqn:pst2}
\begin{aligned}
\intj \pst u \, \vh dx &=\intj u \,\vh dx \quad & &\forall \vh \in P^{k-2}(I_j), \\
(\pst u)^+ +\frac{\ot+\ao}{\bo}(\pst u)_x^+ &=u+\frac{\ot+\ao}{\bo}u_x  & &\textrm{at} \, x_{\jm}, \\
(\pst u)^- -\frac{\ot-\ao}{\bo}(\pst u)_x^- & = u-\frac{\ot-\ao}{\bo}u_x  & &\textrm{at} \, x_{\jp}.
\end{aligned}
\eeq

\item if $\bt \neq 0$, then at $x_{\jp}$ for all $j,$ we have
\begin{equation}
\label{eqn:cancel2}
\begin{aligned}
u_x+\frac{\ot-\ao}{\bt}u & = (\pst u)_x^+ +(\bo+\frac{(\ot-\ao)^2}{\bt})(\pst u)^+ = (\pst u)_x^+ + \frac{\ot-\ao}{\bt}(\pst u)^+, \\
u_x-\frac{\ot+\ao}{\bt}u & = (\pst u)_x^- -(\bo+\frac{(\ot+\ao)^2}{\bt})(\pst u)^- = (\pst u)_x^- -\frac{\ot+\ao}{\bt}(\pst u)^-,
\end{aligned}
\end{equation}

meaning that  $\pst$ can be defined element-wise on cell $I_j$ as:
\beq
\label{eqn:pst3}
\begin{aligned}
\intj \pst u \, \vh dx &=\intj u \,\vh dx \quad & &\forall \vh \in P^{k-2}(I_j), \\
 (\pst u)_x^+ + \frac{\ot-\ao}{\bt}(\pst u)^+ &=u_x+\frac{\ot-\ao}{\bt}u  & &\textrm{at} \, x_{\jm}, \\
(\pst u)_x^- -\frac{\ot+\ao}{\bt}(\pst u)^-& = u_x-\frac{\ot+\ao}{\bt}u  & &\textrm{at} \, x_{\jp}.
\end{aligned}
\eeq

\item if $\bo=\bt=0$, then $\ao=\pm \frac{1}{2}, $ and $\pst=\ppm$ or $\pt$, which are local projections.
\end{itemize}

\end{proof}

This lemma implies that for any parameter satisfying $\ao^2+\bo\bt=\frac{1}{4}$, $\pst$ is locally defined. We remark that this condition turns out to be the same
 as the optimally convergent numerical flux families  in \cite{cheng20152} for two-way wave equations, although they arise in different contexts. Unfortunately, for the general definition of $\pst$, unlike $\ppm$ and $\pt$, we cannot directly use the Bramble-Hilbert lemma and the standard scaling  argument to obtain optimal approximation property, since the second and third relations in \eqref{eqn:pst2} and \eqref{eqn:pst3} may break the scaling.  The next lemma performs a detailed analysis of this local projection when $\bo \neq 0$ or $\bt \neq 0$. Indeed for some parameter choices, only suboptimal convergence rate is obtained.

\begin{Lem}[Local projection: existence, uniqueness and error estimates]
\label{lem:localp}
 If $\ao^2+\bo\bt=\frac{1}{4}$ with $\bo \neq 0$ or $\bt \neq 0$, the   local projection $\pst$ exists and is uniquely defined when 
 \begin{equation}
 \label{eq:localsolvable}
\Gamma_j =  \bo -\frac{k^2}{h_j} + \bt \frac{k^2(k^2-1)}{h_j^2} \ne 0, \, \forall j.
 \end{equation}
  In addition, the following error estimates hold for $p=2, \infty$: %

\begin{eqnarray}
\label{eqn:localpest}
  \|\pst u-u\|_{L^p(I)} \leq  C h^{k+1} |u|_{W^{k+1,\infty}(I)} \left ( 1+ \frac{\max \left( | \bo|,  \min \left (   \frac{|\ot-\ao|}{ h}   ,   \frac{|\ot+\ao|}{  h}   \right ),    \frac{|\bt|}{  h^2}   \right ) }{\min_j |\Gamma_j|} \right ).
\end{eqnarray}

\end{Lem}
\begin{proof}
The proof of this lemma can be found in the Appendix \ref{sec:a01}.
\end{proof}

If we assume $\bo=c/h, \bt=ch$, then $\ao= constant$, and   as long as the solvability condition \eqref{eq:localsolvable} is satisfied,  we have the optimal approximation property for $\pst$. Such conclusions are not surprising, because  \eqref{eqn:pst2} and \eqref{eqn:pst3} will maintain the correct scaling relation.  However, for general parameter choices,  the convergence rate may be suboptimal. This  is verified by numerical experiment in  Table \ref{tab:localp2}.  

\subsubsection{Global projection results}
\label{sec:globalp}

 In this subsection, we consider $\ao^2+\bo\bt \ne \frac{1}{4},$ where $\pst$ is a global projection. For simplicity, only uniform mesh is investigated, which makes the coefficient matrix of the   linear system   block-circulant. First, we analyze the existence and uniqueness of $\pst.$

\begin{Lem}[Global projection: existence and uniqueness]
\label{lem:globalp}
  If $\ao^2+\bo\bt \ne \frac{1}{4}$, assuming a uniform mesh of size $h$, let $\Gamma:=\bo +\frac{k^2(k^2-1)}{h^2}\bt -\frac{2k^2}{h}(\ao^2+\bo \bt+\frac{1}{4})$ and $\Lambda:=\frac{-2k}{h}(\ao^2+\bo\bt-\frac{1}{4})$, then we have

\underline{Case 1.} if $|\Gamma|>|\Lambda|$, then $\pst$  exists and is uniquely defined.

\underline{Case 2.} if $|\Gamma|=|\Lambda|,$  then $\pst$  exists and is uniquely defined if $N$ is odd, and furthermore, if $k$ is odd, we require  $\Gamma=-\Lambda;$ if $k$ is even, we require $\Gamma=\Lambda.$

\underline{Case 3.} if $|\Gamma|<|\Lambda|,$  then $\pst$  exists and is uniquely defined if
$$
(-1)^{(k+1)N}\left(\frac{\Gamma}{\Lambda}+\sqrt{\left (\frac{\Gamma}{\Lambda} \right)^2-1}\right)^N \ne 1.
$$

\end{Lem}

\begin{proof}
The proof of this lemma can be found in the Appendix \ref{sec:a02}.
\end{proof}

\medskip

Next, we will focus on error estimates of the projection $\pst$ based on the three cases as categorized in Lemma \ref{lem:globalp}.
 
\begin{Lem}[Global projection: error estimates for Case 1]
\label{lem:globalpcase1}
  When   the parameter choice belongs to Case 1 in Lemma \ref{lem:globalp}, we have for $p=2, \infty,$
\begin{eqnarray}
\label{eqn:estimate1}
\|\pst u - u\|_{L^p(I)} &\le& C h^{k+1}|u|_{W^{k+1, \infty}(I)} \Bigg ( 1+ \Big ( \frac{|\lt|+1}{|\lt|-1} \left ( \|Q_1 V_1\|_\infty + h^{-1} \|Q_1 V_2\|_\infty \right ) \notag \\
&& +  \frac{1}{|\lt|-1} \left ( \|V_1\|_\infty + h^{-1} \|V_2\|_\infty \right ) \Big ) \Bigg ), \quad \textrm{if}\  \Gamma<0, \notag \\
\|\pst u - u\|_{L^p(I)} &\le& C h^{k+1}|u|_{W^{k+1, \infty}(I)} \Bigg ( 1 + \Big ( \frac{|\lo|+1}{|\lo|-1} \left ( \|(I_2-Q_1) V_1\|_\infty + h^{-1} \|(I_2-Q_1) V_2\|_\infty \right ) \notag \\
&& +  \frac{1}{|\lo|-1} \left ( \|V_1\|_\infty + h^{-1} \|V_2\|_\infty \right ) \Big ) \Bigg ), \quad \textrm{if}\  \Gamma>0,
\end{eqnarray}
where  $Q_1$ is given by \eqref{eqn:q1} or \eqref{eqn:q12} depending on the   parameter choices as shown in the proof; $I_2$ is the $2 \times 2$ identity matrix; $V_1,V_2$ are given by \eqref{eqn:v12}; and $\lo, \lt$ are the  eigenvalues of $Q$ as defined in \eqref{eqn:eig2}. 

\end{Lem}

\begin{proof}
The proof of this lemma can be found in the Appendix \ref{sec:a03}.
\end{proof}

\bigskip
 \eqref{eqn:estimate1} provides  error bound  that can be computed once the parameters $\ao, \bo, \bt$ are given, yet its dependence on  the mesh size $h$ is not fully revealed, particularly when the parameters  $\ao, \bo, \bt$ also have $h$-dependence. To clarify such relations,
next we will consider the following common choice of parameters, where $\ao$ has no dependence on $h$, $\bo=\tilde{\bo} h^{A_1}, \bt=\tilde{\bt} h^{A_2},\tilde{\bo}, \tilde{\bt}$ are nonzero constants that do not depend on $h.$  If indeed $\bo$ or $\bt$ is zero, it is equivalent to let $A_1, A_2 \rightarrow + \infty$ in the discussions below.   We will discuss if the parameter choice will yield optimal $(k+1)$-th order accuracy. To distinguish different cases,  we illustrate the choice  of parameters $A_1, A_2$ in  Figure \ref{fig:case1}.  For example, Case 1.1 means $A_1 > -1, A_2 >1$, Case 1.5 means $A_1 =  -1, A_2 = 1$ and Case 1.7.1 means $A_1 >-1, A_2 = 1$. 
The main results are summarized in  Algorithm \ref{alg:case1}.  

\begin{figure}[h]
\centering
\includegraphics[width = 0.7 \textwidth]{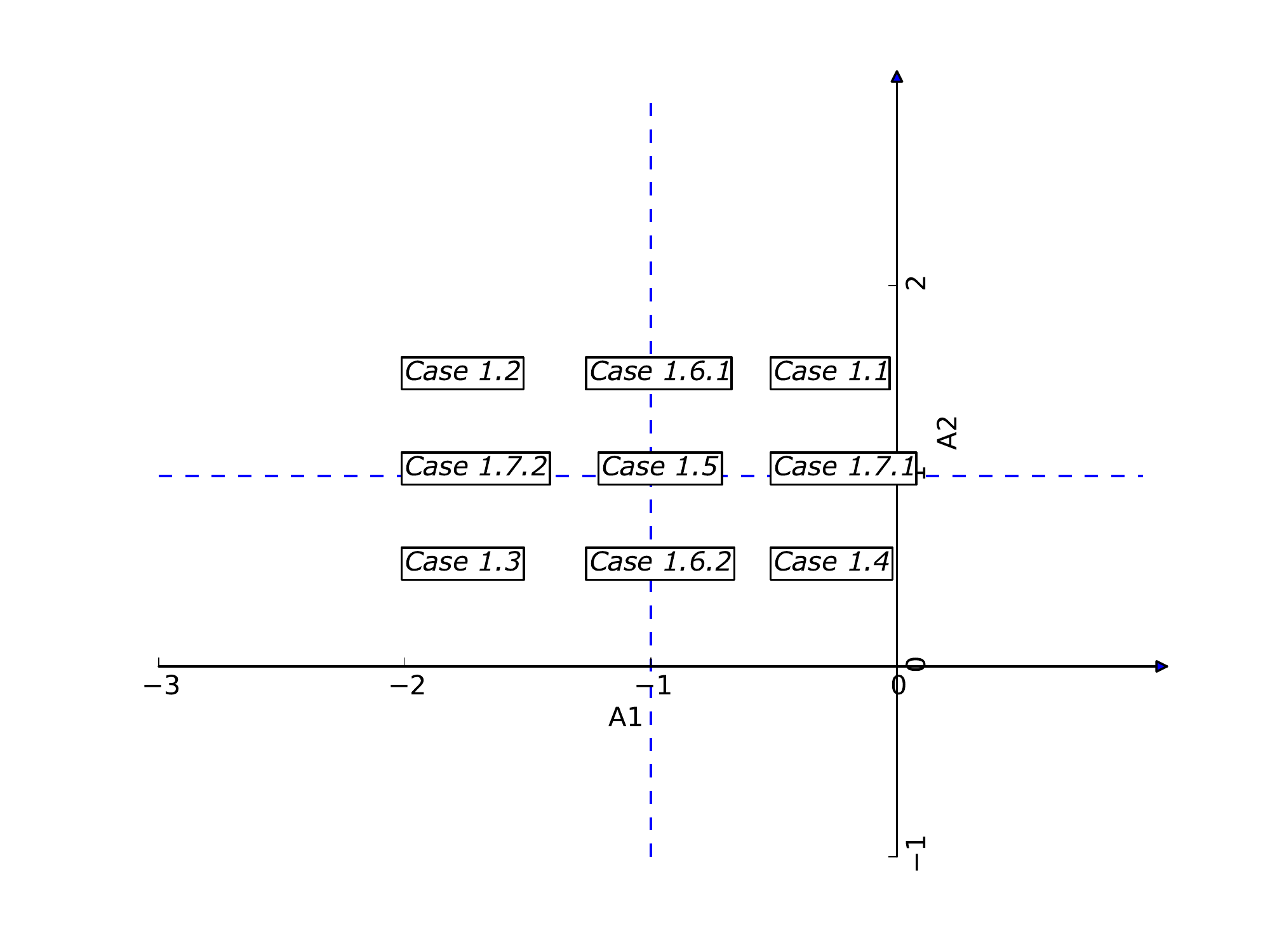}
\caption{A sketch to illustrate the different cases parameterized by the values of $A_1, A_2$.}
\label{fig:case1}
\end{figure}

\medskip

\begin{algorithm}[H]
\SetKwData{Left}{left}\SetKwData{This}{this}\SetKwData{Up}{up}
\SetKwFunction{Union}{Union}\SetKwFunction{FindCompress}{FindCompress}
\SetKwInOut{Input}{Input}\SetKwInOut{Output}{Output}
\SetAlgoLined
\caption{Interpretation of error estimate \eqref{eqn:estimate1}.}
\label{alg:case1}
\BlankLine
\eIf {$k=1$ and  $A_2 <1,$}
{
  $\pst$ is suboptimal and is $(k+A_2)$-th order accurate,  \label{1}%
}
{
	\eIf {$\lim_{h\to 0} \abs{\lo,\lt} = 1$ with $ \left | \lo,\lt \right | = 1 + O(h^{\delta/2}),$}
	{	
	$\pst$ is suboptimal and is $(k+1-\delta)$-th order accurate, \label{3}
	}
	{
	$\pst$ has optimal $(k+1)$-th order error estimates.
	}
}
\end{algorithm}

The main reason of order reduction for  $k=1, A_2 <1$ in  Statement \ref{1} (i.e. line 2 of the algorithm above) is that the term such as $\frac{1}{\abs{\lt}-1}\|Q_1V_1\|_\infty$ is of  $O(h^{A_2-1})$ instead of $O(1)$, and this will cause $(1-A_2)$-th order reduction. 
This happens for   Cases 1.3,  1.4 and    1.6.2 when $k=1.$
The main reason of order reduction    in  Statement \ref{3} is because of the terms such as $\frac{1}{|\lt|-1}, \frac{|\lt|+1}{|\lt|-1}$ in \eqref{eqn:estimate1}.
The fractions $\frac{1}{|\lt|-1}, \frac{|\lt|+1}{|\lt|-1}$ cannot be controlled by a constant if $\lim_{h \to 0} |\lt| = 1$. 
By definition of $\lo,\lt$ in \eqref{eqn:eig2}, we know that $\left |\frac{\g}{\la} \right | \to 1 \Leftrightarrow \left | \lo,\lt \right | \to 1$. More precisely, if $ \left |\frac{\g}{\la} \right | = 1 + O(h^{\delta}), \delta >0$, then $ \left | \lo,\lt \right | = 1 + O(h^{\delta/2}),$ then $\frac{|\lt| +1}{|\lt| -1}$ or $\frac{1}{|\lt|-1} = O(h^{-\delta/2})$. The relation $\g^2 - \la^2 = (b_1 - b_2)(b_1+b_2) + c_2^2$ also indicates that there is some cancellation of leading terms in $b_1 - b_2 $ or $b_1 + b_2$, making $\|Q_1\|_\infty \sim O(h^{-\delta/2})$, multiplying these factors together will result in  $\delta$-th order reduction in the error estimation of $\pst$. Note that $b_1, b_2, c_2$ and $Q_1$ are defined in \eqref{eqn:b1}, \eqref{eqn:b2}, \eqref{eqn:c1} and \eqref{eqn:q1}.

Then we look at what parameter choices  make $ \left |\frac{\g}{\la} \right | \to 1$. Since%
\begin{align*}
\frac{\Gamma}{\Lambda} & = \begin{cases} 
 k + \frac{\bo + \frac{k^2(k^2-1)}{h^2} \bt - \frac{k^2}h}{\Lambda} & k>1,\\
 1 + \frac{\bo - \frac{1}h}{\Lambda} 	& k=1,
 \end{cases}
\end{align*}
 we have  

\begin{enumerate}
	\item \underline{Case 1.1} ($A_1>-1, A_2>1$) with $k=1,\ao = 0$,$\left | \frac{\Gamma}{\Lambda} \right | \to \left |\frac{\ot + 2\ao^2}{\ot - 2\ao^2} \right | = 1.$  
	\item \underline{Case 1.6.1} ($A_1=-1,A_2 >1$) $\tilde \bo =  \frac{k(k\pm1)}{2} + 2\ao^2 k(k\mp1), \abs{\frac{\g}{\la}} \to \abs{k+\frac{\bo - \frac{k^2}{h}}{\la}} \to 1$. 
	\item \underline{Case 1.6.2} ($A_1=-1,A_2 <1$) with $k>1,$ $ \tilde \bo = \frac{k(k\pm 1)}{2},\abs{\frac{\g}{\la}} \to \abs{k+\frac{\frac{k^2(k^2-1)\bt}{h^2}}{\la}}\to 1$.
	\item  \underline{Case 1.7.1} ($A_1>-1,A_2=1$) $ \tilde \bt = \frac{1}{2k(k\mp1)} + \frac{2\ao^2}{k(k\pm1)}, \abs{\frac{\g}{\la}} \to \abs{k+\frac{\frac{k^2(k^2-1)\bt}{h^2} - \frac{k^2}{h}}{\la}} \to 1.$
	\item \underline{Case 1.7.2} ($A_1<-1,A_2=1$)  $\tilde \bt = \frac{1}{2k(k\pm 1)}, \abs{\frac{\g}{\la}} \to \abs{k+\frac{\bo}{\la}} \to 1$.
	\end{enumerate}

\begin{rem}
We only considered $T$  given by \eqref{eqn:T} in the discussion above.   By Appendix \ref{sec:a1}, we can conclude that under the parameter conditions in Case 1, $(b_1+b_2)(b_1-b_2) = 0$ only can happen if $A_1=-1,A_2=1$ with \eqref{eqn:b1mb2} or \eqref{eqn:b1pb2}.  This is Case 1.5, for which we always have optimal error estimate.
\end{rem}

\begin{rem}
\label{rem:case12}
Through numerical tests, we found that
 \eqref{eqn:estimate1} is mostly sharp with two exceptions. When $\lim_{h\to0} | \lo,\lt | = 1$, the estimates show that there will be order reduction for error of $\pst$, while in numerical experiments (see e.g. Tables \ref{tab:case11}, \ref{tab:case12}), such order reduction is observed only when $\lim_{h\to0} \lo,\lt = 1$ but not $-1$. We believe when $\lim_{h\to0} \lo,\lt = -1,$ a refined estimate can be obtained similar  to Lemma \ref{lem:globalpcase2optimal} for Case 2. We have not carried out this estimate in this work.

    Another   example we find for which  \eqref{eqn:estimate1} is not sharp  is $k=2,A_1=-2,-3,A_2=1,  (\ao,\tilde \bo,\tilde \bt) =(0.25, -1, \frac{1}{12})$, where parameters belong to Case 1.7.2, $\tilde \bt = \frac{1}{2k(k+1)}$ and $\lo,\lt \to 1+ O(h^{-(1+A_1)/2})$. The theoretical results predict accuracy order of $(k+2+A_1)$ but numerical experiments in Table \ref{tab:case13} show the order to be $(k+3+A_1)$. 
  Our estimations can't resolve this one order difference. This special parameter may trigger a cancellation we didn't capture in analysis. We will improve this estimate in our future work.   
\end{rem}

We can then generalize the approach to Cases 2 and 3.

\begin{Lem}[Global projection: error estimates for Case 2]
\label{lem:globalpcase2}
 When   the parameter choice belongs to Case 2 in Lemma \ref{lem:globalp} and $\pst$ is well defined, we have
\beq
\label{eqn:estimate2}
\|\pst u -  u\|_{L^p(I)} \le
C  h^{k+1} |u|_{W^{k+1,\infty}(I)} \Bigg(1+ h^{-1} \left (1 +     \frac{ h^{-1} \|Q_2\|_\infty}{ |\Gamma|} \right ) \left ( \|V_1\|_\infty + h^{-1} \|V_2\|_\infty \right ) \Bigg ),
\eeq
where $p=2, \infty,$ $Q_2$ is given by \eqref{eqn:q22} and $V_1,V_2$ are given by \eqref{eqn:v12}.  \end{Lem}

\begin{proof}
The proof of this lemma can be found in the Appendix \ref{sec:a03}.
\end{proof}

\begin{rem}
\label{rem:globalp2}
Detailed discussions on the parameter choices for Case 2 are contained in Appendix \ref{sec:a2}. Under these conditions, we actually have $\g = C \left ( \bo -\frac{k^2}{h} + \frac{k^2(k^2-1)}{h^2}\bt \right ),$ and by \eqref{eqn:v12est} \beq
\label{eqn:case2v12}
 \|V_1\|_\infty + h^{-1} \|V_2\|_\infty \sim C \left (1 + \frac{\max(\abs{\bo}, \abs{\ot - \ao}/h)}{\abs{\g}} \right),
\eeq 
in addition
\beq
\label{eqn:case2q2}
\frac{\|Q_2\|_\infty}{\abs{\g}} \sim C \frac{\max \left ( \abs{\bo}, \oh, \frac{\abs{\bt}}{h^{2}} \right )}{\left | \g \right |}.
\eeq
In the best-case scenario, the right hand side of the two equations above are bounded by a constant.  Therefore, \eqref{eqn:estimate2}  yields the accuracy order to be $(k-1)$ at best. 
\end{rem}

\begin{Lem}[Global projection: error estimates for Case 3]
\label{lem:globalpcase3}
 When   the parameter choice belongs to Case 3 in Lemma \ref{lem:globalp} and $\pst$ is well defined, assuming $\abs{1-\lo^N} \sim O(h^{\delta'}), $ we have
\begin{eqnarray}
\label{eqn:estimate3}
\|\pst u -  u\|_{L^p(I)}  \le 
 C  h^{k+1} |u|_{W^{k+1, \infty}(I)}  \Big(1+  h^{-(\delta'+1)} {\|Q_1\|_\infty} 
  \big ( \|V_1\|_\infty + h^{-1} \|V_2\|_\infty \big ) \Big )
\end{eqnarray}
where $p=2, \infty$ and $Q_1,V_1,V_2$ are given by \eqref{eqn:q1} and \eqref{eqn:v12}.
\end{Lem}

\begin{proof}
The proof of this lemma can be found in the Appendix \ref{sec:a03}.
\end{proof}

\begin{rem}
In the best-case scenario, the term $\|Q_1\|_\infty$ and $ \|V_1\|_\infty + h^{-1} \|V_2\|_\infty$ are bounded by constants. While the term  $h^{-(\delta'+1)}$ is of order at least $h^{-1}$, leading to loss of at least one order of accuracy.  
\end{rem}

Lemmas \ref{lem:globalpcase2} and  \ref{lem:globalpcase3} only give suboptimal results. In what follows, we aim at improving the convergence order    with stronger assumption on the regularity of  the solution   by using  additional techniques involving cancellation of errors from neighboring terms and global approximation by Fourier expansions. 
We will need the following lemma that resembles  Proposition 3.2 in \cite{XingKdvDG}, and also the fast decay property of Fourier coefficients of the exact solution. The proof of Lemma \ref{lem:rep} follows the same line as in  \cite{XingKdvDG} and is skipped for brevity. %

\begin{Lem} {\rm (Detailed error estimates for $P_h^1$)}
\label{lem:rep}
When $P_h^1$ is applied to a periodic and sufficiently smooth function $u$ on uniform mesh, denote $\eta_j = (u-P_h^1 u)^+|_{j+\ot}$ and $\theta_j = (u_x-(P_h^1 u)_x)^-|_{j+\ot}, \ \ j=0,\cdots, N-1$, we have:  
\begin{eqnarray}
\label{lem:rep1}
\eta_{j-1} & = & \mu  h^{k+1} u^{(k+1)} (x_{j-\ot}) + \mu_2  h^{k+2} u^{(k+2)}(x_{\jm}) + C_2 h^{k+3}, \\
\label{lem:rep2}
\theta_j & = & \rho  h^k u^{(k+1)} (x_{j-\ot}) + \rho_2   h^{k+1} u^{(k+2)}(x_{\jm}) +C_3 h^{k+2},
\end{eqnarray}
where $\mu,\mu_2,\rho$ and $\rho_2$ are constants that depend only on $k.$ $C_2$ and $C_3$ depend on $k$ and $| u|  _{W^{k+3,\infty}(I_j)}$. Thus, by using Mean-Value Theorem, an additional $h$ can be extracted,
\begin{eqnarray}
|\eta_j-\eta_{j+1}| & \leq & C h^{k+2} | u|  _{W^{k+2,\infty}(I)}, \\
|\theta_j-\theta_{j+1}| & \leq & C h^{k+1} | u|  _{W^{k+2,\infty}(I)}. \label{eqn:thetaest}
\end{eqnarray}
\end{Lem}

With Lemma \ref{lem:rep} and Fourier analysis, we can prove the following two lemmas with refined error estimates.

\begin{Lem}[Global projection: refined error estimates for Case 2]
\label{lem:globalpcase2optimal}
 When   the parameter choice belongs to Case 2 in Lemma \ref{lem:globalp} and $\pst$ is well defined, we have
\beq
\label{eqn:estimate2optimal}
\|\pst u - u\|_{L^p(I)} \le C h^{k+1} \|u\|_{W^{k+4,\infty}(I)} \left ( 1+ \left (1+\frac{ \norm{Q_2}_\infty}{|\Gamma|} \right ) (\|V_1\|_\infty + h^{-1} \|V_2 \|_\infty) \right ),
\eeq
where $p=2, \infty,$ $Q_2$ is given by \eqref{eqn:q22}, $V_1,V_2$ are given by \eqref{eqn:v12}.
\end{Lem}

\begin{proof} 
The proof of this lemma can be found in the Appendix \ref{sec:a04}.
\end{proof}

\begin{rem}
\label{rem:case2}
The difference between \eqref{eqn:estimate2optimal} and \eqref{eqn:estimate2} are the two $h^{-1}$ factors and the norm of $u$, which corresponds to the different regularity requirement for the estimation. It is obvious that \eqref{eqn:estimate2optimal} is always a better estimate if the solution is smooth enough.

In most cases,  \eqref{eqn:estimate2optimal} yields optimal accuracy order, except when  $k=1, \ao = 0, \bo = 0, \bt=O(h^{A_2}), A_2 <1$, where the $\pst$ is only $(k+A_2)$-th order accurate because $\frac{\|Q_2\|_\infty}{|\la|} = \frac{|b_1+b_2|}{|\la} = \frac{|-\frac{-4}{h^2}\bt + \frac{1}{2h}|}{\frac{1}{2h}} \sim O(h^{A_2-1})$ in \eqref{eqn:estimate2optimal}. %
This is verified numerically in Table \ref{tab:central2}.%
\end{rem}

\begin{Lem}[Global projection: refined error estimates for Case 3]
\label{lem:globalpcase3optimal}
 When the parameter choice belongs to Case 3 in Lemma \ref{lem:globalp} and $\pst$ is well defined, assuming $\abs{1-\lo^N}=O(h^{\delta'})$ and $|\lo - 1|=O(h^{\delta/2})$ with $0 \leq \delta/2 \leq 1, $ we have
\beq
\label{eqn:estimate3optimal}
\|\pst u - u\|_{L^p(I)}  \le 
C h^{k+1} \|u\|_{W^{k+3, \infty}(I) } \left (1+ h^{-(\delta'+\delta/2)}  \|Q_1\|_\infty  (\|V_1\|_\infty + h^{-1} \|V_2\|_\infty) \right ) ,
\eeq
where $p=2, \infty,$ $\lo$ is the eigenvalue of $Q$ defined in \eqref{eqn:eig2}, $Q_1$ is given by \eqref{eqn:q1}, $V_1,V_2$ are given by \eqref{eqn:v12}.  
\end{Lem}

\begin{proof}
The proof of this lemma can be found in the Appendix \ref{sec:a05}.
\end{proof}

\begin{rem}
\label{rem:case3} 
If $0 \leq \delta/2 \leq 1$, Lemma \ref{lem:globalpcase3optimal} is always a better estimate than Lemma \ref{lem:globalpcase3} when the solution is smooth enough. If $\delta/2 > 1$, we can show $\delta/2 = \delta' +1$. This is because $|1-\lo|=|1-e^{i\theta}| = 2|\sin(\theta/2)|,$ and $|1-\lo^N|=|1-e^{iN\theta}| = 2|\sin(N\theta/2)|.$ When $\delta/2>1,$ one can assert that $|1-\lo|\sim \theta, |1-\lo^N|\sim N\theta,$  i.e. $\delta/2 = \delta' +1.$ With this condition, we notice that Lemma \ref{lem:globalpcase3} yields an reduction of $\delta$-th order in convergence rate by checking the order of each term as is done for Case 1. This order reduction   is consistent with numerical experiments in Example \ref{exa:globalp3}. Therefore, there is no need to further improve the estimates as is done for $0 \leq \delta/2 \leq 1$ in Lemma \ref{lem:globalpcase3optimal}.

\end{rem}

Now we can summarize the estimation of $\pst$  for some frequently used flux parameters. For IPDG scheme with $\ao = \bt = 0, \bo = c/h$, and DDG scheme discussed in \cite{luddgsch} with $\ao = constant, \bo = c/h, \bt = 0$,  and the more general scale invariant parameter choice $\ao = constant, \bo = c/h, \bt = ch$, $\pst$ always have optimal error estimates. For those parameters, we can show that the eigenvalues $\lo, \lt$ are always   constants independent of $h$, therefore, either by estimates for local projection in Lemma \ref{lem:localp} or global projection in  Lemmas \ref{lem:globalpcase1}, \ref{lem:globalpcase2optimal}, \ref{lem:globalpcase3optimal}, we will  have optimal convergence rate. Corresponding numerical results are shown in Tables \ref{tab:localp2} and \ref{tab:case15}. 

For a natural parameter choice where $\ao,   \bo, \bt$ are all   real constants, if $\bt \neq 0$, then $\pst$ has first order convergence rate when $k=1$ and optimal convergence rate when $k>1$ by Lemmas \ref{lem:localp}, \ref{lem:globalpcase1}, \ref{lem:globalpcase2optimal}, \ref{lem:globalpcase3optimal}. Corresponding numerical results are shown in Tables \ref{tab:localp} and \ref{tab:central2}. Lastly, for central flux $\ao=\at=\bo=\bt= 0$, this parameter choice belongs to Case 2 when $k=1$ and Case 1 when $k>1$, thus we can verify that $\pst$ has optimal convergence rate by Lemmas   \ref{lem:globalpcase1} and \ref{lem:globalpcase2optimal}. Corresponding numerical results are shown in Table \ref{tab:central}.

\subsection{Error estimates of the DG scheme }
\label{sec:l2estimate}

We are now ready to state the main theorem, which is the semi-discrete $L^2$ error estimates of the DG scheme \eqref{eqn:scheme} with numerical flux \eqref{eqn:flux2}.
\begin{Thm}
\label{thm:convergence}
Assume that the exact solution $u$ and the nonlinear term $f(|u|^2)$ of 
\eqref{eqn:nls1} are sufficiently smooth with bounded derivatives for any time $t\in (0,T_e]$ and that the numerical flux parameters in \eqref{eqn:flux2} satisfy the existence conditions of $\pst$ in Lemmas \ref{lem:localp} or \ref{lem:globalp}. Furthermore, assume $\epsilon_h = u - \pst u$ has at least first order convergence rate in $L^2$ and $L^\infty$ norm from the results in Section \ref{sec:projection}. With periodic boundary conditions, uniform mesh size and solution space $V_h^k$ $(k\geq 1)$, the following error estimation holds for $u_h$, which is the numerical solution of \eqref{eqn:scheme} with flux \eqref{eqn:flux2}:
\begin{align}
\label{eqn:errscheme}
\|u-u_h\|_{L^2({I})} %
 \leq C_\star  \left (\| (u - u_h)|_{t=0} \|_{L^2(I)}  +  \|(\epsilon_h)_t\|_{L^2(I)} + \|\epsilon_h\|_{L^2(I)}  \right ),
\end{align}
where $C_\star$ depends on $k,\|f\|_{W^{2,\infty}},$ $u$   as well as final time $T_e$, but not on $h$. In other words, the error of the DG scheme \eqref{eqn:scheme} has same order of convergence rate as the projection $\pst$ in Lemmas \ref{lem:localp}, \ref{lem:globalpcase1}-\ref{lem:globalpcase3optimal} depending on the parameter choices, if the numerical initial condition is chosen sufficiently accurate.  
\end{Thm}

\begin{proof}
When $\pst$ exists, we can decompose the error into two parts.
\[
e=u-u_h=u-\pst u + \pst u - u_h : = \epsilon_h + \zeta_h.
\]

By Galerkin orthogonality
\begin{align*}
0 &=  a_{\ao, -\ao, \bo, \bt}(e,v_h)-i \int_I f(|u|^2)uv_h dx+i \int_I f(|u_h|^2)u_hv_h dx    \qquad \forall v_h\in V_h^k\\
&=  a_{\ao, -\ao, \bo, \bt}(\epsilon_h, v_h)+ a_{\ao, -\ao, \bo, \bt}(\zeta_h, v_h)-i \int_I f(|u|^2)uv_h dx+i \int_I f(|u_h|^2)u_hv_h dx.
\end{align*}

Let $v_h = \overline{\zeta_h}$, and take conjugate of above equation, we have
\begin{align}
\label{eqn:err}
&a_{\ao, -\ao, \bo, \bt}(\zeta_h, \overline{\zeta_h})+\overline{a_{\ao, -\ao, \bo, \bt}(\zeta_h, \overline{\zeta_h})}\\
=&  -a_{\ao, -\ao, \bo, \bt}(\epsilon_h, \overline{\zeta_h})-\overline{a_{\ao, -\ao, \bo, \bt}(\epsilon_h, \overline{\zeta_h})} -2\int_I f(|u|^2)\mathrm{Im}(u \overline{\zeta_h})dx+ 2 \int_I f(|u_h|^2)\mathrm{Im}(u_h \overline{\zeta_h}) dx. \notag
\end{align}

By Taylor expansion
\[
f(|u_h|^2) = f(|u|^2) + f'(|u|^2)E + \ot \hat f'' E^2,
\]
where $\hat f'' = f''(c), c$ is a value between $|u_h|^2 $ and $|u|^2.$ $E = |u_h|^2 - |u|^2 = -2\mathrm{Re}(e\overline u) + |e|^2.$
Therefore, the nonlinear part becomes
\begin{align*}
&\int_I f(|u|^2)\mathrm{Im}(u \overline{\zeta_h})dx - \int_I f(|u_h|^2)\mathrm{Im}(u_h \overline{\zeta_h}) dx\\
=&\int_I f(|u_h|^2)\mathrm{Im}\big ( e \overline{\zeta_h} \big) +  \big ( f(|u|^2) - f(|u_h|^2) \big ) \mathrm{Im}(u \overline{\zeta_h}) dx \\
=& \mathcal{N}_1 +  \mathcal{N}_2 + \mathcal{N}_3,
\end{align*}
where
\begin{align*}
& \mathcal{N}_1 = \int_I f(|u|^2) \mathrm{Im}\big ( e \overline{\zeta_h} \big) - f'(|u|^2)E \mathrm{Im}(u \overline{\zeta_h}) dx,\\
& \mathcal{N}_2 = \int_I f'(|u|^2)E \mathrm{Im}\big ( e \overline{\zeta_h} \big) - \ot \hat f'' E^2 \mathrm{Im}(u \overline{\zeta_h}) dx,\\
& \mathcal{N}_3 = - \int_I \ot \hat f'' E^2 \mathrm{Im}\big ( e \overline{\zeta_h} \big),
\end{align*}
will be estimated separately as follows.
\begin{itemize}
\item $\mathcal{N}_1$ and $\mathcal{N}_2$ terms.

Since $e\overline{\zeta_h} = \epsilon_h \overline{\zeta_h} + \abs{\zeta_h}^2$, $\abs{E \mathrm{Im} (u \overline{\zeta_h})} =\abs{(-2 \mathrm{Re} (e\overline u) +\abs{e}^2 )\mathrm{Im} (u \overline{\zeta_h}) } \leq C  ( \|u\|_{L^\infty (I)}^2 + \|u\|_{L^\infty (I)} \|e\|_{L^\infty(I)} )( \|\epsilon_h\|^2_{L^2 (I)} + \|\zeta_h\|^2_{L^2 (I)})$, we have
\begin{align*}
	\abs{\mathcal{N}_1}&  \leq C  \|f\|_{W^{1,\infty}} \left ( 1 +  \|u\|_{L^\infty (I)}^2 + \|u\|_{L^\infty (I)} \|e\|_{L^\infty(I)}  \right ) ( \|\epsilon_h\|^2_{L^2 (I)} + \|\zeta_h\|^2_{L^2 (I)}), \\
	  \abs{\mathcal{N}_2} & \leq C  \|f\|_{W^{2,\infty}} \|E\|_{L^\infty(I)} \left ( 1 +  \|u\|_{L^\infty (I)}^2 + \|u\|_{L^\infty (I)} \|e\|_{L^\infty(I)}  \right ) ( \|\epsilon_h\|^2_{L^2 (I)} + \|\zeta_h\|^2_{L^2 (I)}).
\end{align*}
\item $\mathcal{N}_3$ term.
$$
\abs{\mathcal{N}_3} \leq C  \|f''\|_{L^{\infty}} \|E\|_{L^\infty(I)}^2   ( \|\epsilon_h\|^2_{L^2 (I)} + \|\zeta_h\|^2_{L^2 (I)}).
$$
\end{itemize}

To conduct a proper estimate for the nonlinear part, we would like to make an \emph{a priori} assumption that, for $h$ small enough, 
\beq
\| e\|_{L^2(I)} = \| u - u_h \|_{L^2(I)} \leq   h^{0.5}.
\eeq
 By our assumption on $\pst$,  $\|{\epsilon_h}\|_{L^p(I)} \leq C_1h, p = 2,\infty$,   thus $\| \zeta_h \|_{L^2(I)} \leq C_1 h^{0.5}$ and $\| \zeta_h \|_{L^\infty(I)} \leq C_1$ by inverse inequality, then $\| e \|_{L^\infty(I)} \leq C_1$, $\|E\|_{L^{\infty}(I)} \leq C_1$.  Here and below, $C_1$  is a generic constant that has no dependence on $h$, but may depend on $u$ according to the lemma used to estimate $\epsilon_h.$

Therefore, we get the estimate:
\beq
\label{eqn:nonlinearest}
\abs{\mathcal{N}_1} + \abs{\mathcal{N}_2} + \abs{\mathcal{N}_3} \leq C_1 ( \|\epsilon_h\|^2_{L^2 (I)} + \|\zeta_h\|^2_{L^2 (I)}),
\eeq
where $C_1$ depends on  $ \|f\|_{W^{2,\infty}}$ and $u.$

For linear part of the right hand side in \eqref{eqn:err}, we have
\begin{align*}
a_{\ao, -\ao, \bo, \bt}(\epsilon_h, \overline{\zeta_h}) + \overline{a_{\ao, -\ao, \bo, \bt}(\epsilon_h, \overline{\zeta_h})} & =  \int_I (\epsilon_h)_t \overline{\zeta_h} + \overline{(\epsilon_h)_t} \zeta_h dx -i \int_I (\epsilon_h) (\overline{\zeta_h})_{xx} dx \\
 &+i \int_I \overline{(\epsilon_h)} ({\zeta_h})_{xx} dx  -i\sumj (\widehat{\epsilon_h} [(\overline{\zeta_h})_x]-\widetilde{(\epsilon_h)}_x[\overline{\zeta_h}])|_\jp  \\
                                                         &+ i\sumj \overline{(\widehat{\epsilon_h} [(\overline{\zeta_h})_x]-\widetilde{(\epsilon_h)}_x[\overline{\zeta_h}])|_\jp}, \\
                                                         & = 2\int_I \mathrm{Re}\big ((\epsilon_h)_t \overline{\zeta_h} \big )dx.
\end{align*}
The last equality holds because of the definition of $\pst u.$
For the left hand side of \eqref{eqn:err}, by similar computation in stability analysis we have
\beq
a_{\ao, -\ao, \bo, \bt}(\zeta_h, \overline{\zeta_h})+\overline{a_{\ao, -\ao, \bo, \bt}(\zeta_h, \overline{\zeta_h})} = \frac{d}{dt} \int_I |\zeta_h|^2 dx.
\eeq

Combine these two equations with \eqref{eqn:nonlinearest}:
\[
\frac{d}{dt} \|\zeta_h\|^2_{L^2(I)} \leq  \| (\epsilon_h)_t\|_{L^2(I)}^2 + \|\zeta_h\|_{L^2(I)}^2 + C_1  ( \|\epsilon_h\|^2_{L^2 (I)} + \|\zeta_h\|^2_{L^2 (I)}).
\]
Assuming $u_t, u $ have sufficient smoothness, then by Gronwall's inequality, we can get:
\[
\|\zeta_h\|_{L^2(I)}^2 \leq C_1 \left (\|\zeta_h|_{t=0}\|^2_{L^2(I)} +  \|(\epsilon_h)_t\|_{L^2(I)}^2 + \|(\epsilon_h)\|_{L^2(I)}^2 \right ),
\]
 and we obtain \eqref{eqn:errscheme}.

To complete the proof, we shall justify the \emph{a priori} assumption. To be more precise, we consider $h_0$, s.t., $\forall h < h_0, C_\star h \leq \frac{1}{2}h^{0.5}$, where $C_\star$ is defined in  \eqref{eqn:errscheme}, dependent on $T_e$, but not on $h$. Suppose $\exists \, t^\ast$ = $\sup\{t: \|u( t^\ast) -u_h(t^\ast)\|_{L^2(I)} \} \leq h^{0.5}$, we would have $\|u(t^\ast) -u_h(t^\ast)\|_{L^2(I)} = h^{0.5}$ by continuity if $t^\ast$ is finite. By  \eqref{eqn:errscheme}, we obtain $\|e\|_{L^2(I)} \leq C_\star h \leq \frac{1}{2}h^{0.5}$ if $t^\ast \leq T_e$, which contradicts the definition of $t^\ast$. Therefore, $t^\ast >T_e$ and the \emph{a priori} assumption is justified.

\end{proof}

\begin{rem}
If $f$ is a constant function, we can prove the same error estimates without using the \emph{a priori} assumption. Therefore, the assumption that $\epsilon_h = u - \pst u$ has at least first order convergence rate in $L^2$ and $L^\infty$ norm is no longer needed.
\end{rem}

\section{Numerical experiments}
\label{sec:numerical}
In this section, we present  numerical experiments to validate our theoretical results. Particularly, in Section \ref{sec:numerical1}, we provide numerical validations of convergence rate for the projection $\pst$ as discussed in Section \ref{sec:projection} with focus on  the dependence of the errors   on parameters $\ao, \bo, \bt$ .  Section \ref{sec:numerical2}   illustrates the energy conservation property  and  validates theoretical convergence rate of DG scheme for NLS equation \eqref{eqn:nls1}.

\subsection{Numerical results of the projection operator $\pst$}
\label{sec:numerical1}

\begin{Example} 
\label{exa:localp}
In this example, we focus on local projection where $\ao^2+\bo\bt=\frac{1}{4}$, and verify the conclusions in Lemma \ref{lem:localp} by considering a smooth test function $u= \cos(x)$ on $[0, 2\pi]$ with a uniform mesh of size $h=2\pi/N$ and $k=1, 2, 3$ for various sets of parameters $(\ao,\bo,\bt).$  
\end{Example}

We first   consider two sets of parameters  $(\ao,\bo,\bt) = (0.3,0.4,0.4)$ and $(\ao,\bo,\bt) = (0.3,0.4/h,0.4h).$
The results with $(\ao,\bo,\bt) = (0.3,0.4,0.4)$  are listed in Table \ref{tab:localp}. By plugging in the parameters into \eqref{eqn:localpest}, we have  that when $k=1$, the projection has suboptimal first order convergence rate, while for $k>1$, optimal $(k+1)$-th order convergence rate should be achieved. Results in Table \ref{tab:localp} agree well with the theoretical prediction. On the other hand, when we choose parameters $(\ao,\bo,\bt) = (0.3,0.4/h,0.4h)$, by Lemma  \ref{lem:localp}, we should observe optimal convergence rate for all $k \ge 1$, and this is verified by the numerical results in Table \ref{tab:localp2}.

Then, we choose the parameters as $(\ao,\bo,\bt) = (0.5,1,0)$ to verify the super-closeness claim
\eqref{eqn:superc1},  i.e., the difference between $\pst$ and $P_h^1$ can have convergence rates higher than $k+1$.  The results are listed in Table \ref{tab:diff}. The difference of the two projections is indeed of $(k+2)$-th order for any $k \ge 1$ in all norms. Finally, we take $(\ao,\bo,\bt) =(0.5, \frac{k^2}{h(1+h)},  0)$. In this case, $\Gamma_j=O(1).$  The numerical results in Table \ref{tab:localplei} verify the order reduction to $k$-th order accuracy for all $k \ge 1$ as predicted by \eqref{eqn:localpest}.

\begin{table}[!h]
	\centering
	\small
	\caption{Example \ref{exa:localp}. Error of local projection $\pst u - u$. Flux parameters: $\alpha_1=0.3,  \beta_1=0.4, \beta_2=0.4.$}
	\label{tab:localp}
	\begin{tabular}{|c|c|c|c|c|c|c|c|}
	\hline
	& N   & $L^1$ error & order&$L^2$ error & order& $L^{\infty}$ error & order \\
	\hline
  \multirow{4}{1em}{$ P^1$} 
 & 160&     0.49E-02&   -&     0.27E-01&  -&     0.16E-01&   -\\
 & 320&     0.25E-02&   0.99&     0.14E-01&   0.99&     0.79E-02&   1.00\\
 & 640&     0.12E-02&   0.99&     0.69E-02&   0.99&     0.39E-02&   1.00\\
& 1280&     0.62E-03&   1.00&     0.35E-02&   1.00&     0.20E-02&   1.00\\  
\hline
 \multirow{4}{1em}{$ P^2$} 
& 160&     0.52E-06&   -&     0.32E-05&   -&     0.26E-05&   -\\
& 320&     0.64E-07&   3.01&     0.39E-06&   3.01&     0.32E-06&   3.02\\
& 640&     0.80E-08&   3.01&     0.49E-07&   3.01&     0.40E-07&   3.01\\
&1280&     0.10E-08&   3.00&     0.61E-08&   3.00&     0.49E-08&   3.01\\   
\hline
 \multirow{4}{1em}{$ P^3$} 
& 160&     0.58E-09&   -&     0.39E-08&   -&     0.33E-08&   -\\
& 320&     0.36E-10&   4.00&     0.24E-09&   4.00&     0.21E-09&   4.01\\
& 640&     0.22E-11&   4.00&     0.15E-10&   4.00&     0.13E-10&   4.00\\
&1280&     0.14E-12&   4.00&     0.94E-12&   4.00&     0.80E-12&   4.00\\ 
	\hline
	\end{tabular}
\end{table}

\begin{table}[!h]
	\centering
	\small
	\caption{Example \ref{exa:localp}. Error of local projection $\pst u - u$. Flux parameters: $\alpha_1=0.3,   \beta_1=0.4/h, \beta_2=0.4h.$}
	\label{tab:localp2}
	\begin{tabular}{|c|c|c|c|c|c|c|c|}
	\hline
	& N   & $L^1$ error & order&$L^2$ error & order& $L^{\infty}$ error & order \\
	\hline
  \multirow{4}{1em}{$ P^1$} 
& 160&     0.82E-04&   -&     0.61E-03&   -&     0.74E-03&   -\\
& 320&     0.20E-04&   2.00&     0.15E-03&   2.00&     0.19E-03&   2.00\\
& 640&     0.51E-05&   2.00&     0.38E-04&   2.00&     0.46E-04&   2.00\\
&1280&     0.13E-05&   2.00&     0.95E-05&   2.00&     0.12E-04&   2.00\\ 
\hline
 \multirow{4}{1em}{$ P^2$} 
& 160&     0.14E-05&   -&     0.88E-05&   -&     0.89E-05&   -\\
& 320&     0.17E-06&   3.00&     0.11E-05&   3.00&     0.11E-05&   3.00\\
& 640&     0.22E-07&   3.00&     0.14E-06&   3.00&     0.14E-06&   3.00\\
&1280&     0.27E-08&   3.00&     0.17E-07&   3.00&     0.17E-07&   3.00\\
\hline
 \multirow{4}{1em}{$ P^3$} 
& 160&     0.68E-09&   -&     0.45E-08&   -&     0.43E-08&   -\\
& 320&     0.43E-10&   4.00&     0.28E-09&   4.00&     0.27E-09&   4.00\\
& 640&     0.27E-11&   4.00&     0.18E-10&   4.00&     0.17E-10&   4.00\\
&1280&     0.17E-12&   4.00&     0.11E-11&   4.00&     0.11E-11&   4.00\\
	\hline
	\end{tabular}
\end{table}

\begin{table}[!h]
	\centering
	\small
	\caption{Example \ref{exa:localp}. Difference of local projection $\pst$ with $\ppm$: $\pst u-\ppm u$.  Flux parameters: $\alpha_1=0.5,   \beta_1=1, \beta_2=0.$}
	\label{tab:diff}
	\begin{tabular}{|c|c|c|c|c|c|c|c|}
	\hline
	& N   & $L^1$ error & order&$L^2$ error & order& $L^{\infty}$ error & order \\
	\hline
  \multirow{4}{1em}{$ P^1$} 
& 160&     0.50E-05&   -&     0.32E-04&   -&     0.31E-04&   -\\
& 320&     0.61E-06&   3.03&     0.40E-05&   3.03&     0.38E-05&   3.03\\
& 640&     0.76E-07&   3.01&     0.49E-06&   3.01&     0.47E-06&   3.01\\
&1280&     0.95E-08&   3.01&     0.61E-07&   3.01&     0.58E-07&   3.01\\
\hline
 \multirow{4}{1em}{$ P^2$} 
& 160&     0.12E-08&   -&     0.81E-08&   -&     0.12E-07&   -\\
& 320&     0.75E-10&   4.01&     0.50E-09&   4.01&     0.72E-09&   4.01\\
& 640&     0.46E-11&   4.00&     0.31E-10&   4.00&     0.45E-10&   4.00\\
&1280&     0.29E-12&   4.00&     0.20E-11&   4.00&     0.28E-11&   4.00\\
\hline
 \multirow{4}{1em}{$ P^3$} 
& 160&     0.75E-12&   -&     0.50E-11&   -&     0.80E-11&   -\\
& 320&     0.23E-13&   5.00&     0.16E-12&   5.00&     0.25E-12&   5.00\\
& 640&     0.73E-15&   5.00&     0.49E-14&   5.00&     0.78E-14&   5.00\\
&1280&     0.23E-16&   5.00&     0.15E-15&   5.00&     0.24E-15&   5.00\\
	\hline
	\end{tabular}
\end{table}

\begin{table}[!h]
	\centering
	\small
	\caption{Example \ref{exa:localp}. Error of local projection $\pst u - u$. Flux parameters: $\alpha_1=0.5,  \beta_1=\frac{k^2}{h(1+h)}, \beta_2=0$.   }
	\label{tab:localplei}
	\begin{tabular}{|c|c|c|c|c|c|c|c|}
	\hline
	& N   & $L^1$ error & order&$L^2$ error & order& $L^{\infty}$ error & order \\
	\hline
  \multirow{4}{1em}{$ P^1$} 
& 160&     0.33E-02&   -&     0.21E-01&   -&     0.20E-01&   -\\
& 320&     0.16E-02&   1.04&     0.10E-01&   1.03&     0.98E-02&   1.03\\
& 640&     0.79E-03&   1.02&     0.51E-02&   1.02&     0.49E-02&   1.01\\
&1280&     0.39E-03&   1.01&     0.25E-02&   1.01&     0.24E-02&   1.01\\
\hline
 \multirow{4}{1em}{$ P^2$} 
& 160&     0.33E-05&   -&     0.22E-04&   -&     0.31E-04&   -\\
& 320&     0.79E-06&   2.04&     0.54E-05&   2.04&     0.76E-05&   2.03\\
& 640&     0.20E-06&   2.02&     0.13E-05&   2.02&     0.19E-05&   2.02\\
&1280&     0.49E-07&   2.01&     0.33E-06&   2.01&     0.47E-06&   2.01\\
\hline
 \multirow{4}{1em}{$ P^3$} 
& 160&     0.47E-08&   -&     0.31E-07&   -&     0.49E-07&   -\\
& 320&     0.57E-09&   3.06&     0.38E-08&   3.05&     0.59E-08&   3.03\\
& 640&     0.69E-10&   3.03&     0.46E-09&   3.02&     0.73E-09&   3.02\\
&1280&     0.86E-11&   3.01&     0.57E-10&   3.01&     0.91E-10&   3.01\\
	\hline
	\end{tabular}
\end{table}

\begin{Example}
\label{exa:globalp1}
In this example, we consider global projection when the parameter choices belong to Case 1.  We consider a smooth test function $u = e^{\cos(x)}$ on $[0,2\pi]$ with a uniform mesh of size $h = 2\pi/N$ and $k=1,2,3$ for various sets of parameters $(\ao,\bo,\bt).$
\end{Example}

We first test the situation when $\lim_{h\to0} |\lo,\lt| \neq 1$ by setting the parameters $(\ao,\tilde \bo,\tilde \bt) = (0.25,1,1),A_1 =-0.5,A_2 =2.$  Another example is $(\ao,\bo,\bt) = (0, \frac{1}{2h}, h)$, for which the eigenvalues $\lo,\lt$ are constant dependent on $k$ but not $h$. These two parameter choices belong to Case 1.1 and Case 1.5, respectively. The numerical results shown in Tables \ref{tab:case14} and \ref{tab:case15} verify the optimal $(k+1)$-th order convergence rate predicted by Lemma \ref{lem:globalpcase1}.

Then we test the situation when $\lim_{h\to0} |\lo,\lt| =1$ by using two sets of parameters $(\ao,\tilde \bo,\tilde \bt) = (0.25,\frac{k(k-1)}{2} +  \frac{k(k+1)}{8},1), A_1=-1, A_2=2,3,$ and $(\ao,\tilde \bo,\tilde \bt) = (0.25,\frac{2}{k(k-1)},1), A_1=-2,-3, A_2=1$. The first set of parameters belongs to  Case 1.6.1 and we can verify that $\lim_{h\to0} \lo, \lt = (-1)^k.$ Lemma \ref{lem:globalpcase1} and Algorithm \ref{alg:case1} imply $(k+2-A_2)$-th convergence order. The numerical results listed in Table \ref{tab:case11}  show that the expected order reduction only happens when $\lim_{h\to0} \lo,\lt =1$, but not for $\lim_{h\to0} \lo,\lt =-1.$ The second set of parameters belongs to  Case 1.7.2 and   we can verify that $\lim_{h\to0} \lo, \lt = (-1)^{k+1}.$ Lemma \ref{lem:globalpcase1} and Algorithm \ref{alg:case1} imply  $(k+2+A_1)$-th convergence order. The numerical results listed in Table  \ref{tab:case12} also show that order reduction is only observed when $\lim_{h\to0} \lo,\lt =1$.

Lastly, we test  $(\ao,\tilde \bo,\tilde \bt) =(0.25, -1, \frac{1}{12})$ with $k=2,A_1=-2,-3,A_2=1,$ where our theoretical results predict accuracy order of $(k+2+A_1)$, but   numerical experiments show the order to be $(k+3+A_1)$  in Table \ref{tab:case13}. This is one of the exceptions that Lemma \ref{lem:globalpcase1} is not sharp and has been commented in Remark \ref{rem:case12}. 

\begin{table}[!h]
	\centering
	\small
	\caption{Example \ref{exa:globalp1}. Error of global projection $\pst u - u$. Flux parameters (Case 1.1): $\alpha_1=0.25,  \tilde\beta_1=1 , \tilde\beta_2=1$, $A_1=-0.5, A_2 =2$.}
	\label{tab:case14}
	\begin{tabular}{|c|c|c|c|c|c|c|cH|}
	\hline
	& N   & $L^1$ error & order&$L^2$ error & order& $L^{\infty}$ error & order &\\
	\hline
  \multirow{4}{1em}{$ P^1$} 
& 160&     0.10E-03&  -&     0.69E-03&  -&     0.89E-03&  -&   -1.2308\\
& 320&     0.26E-04&   1.93&     0.18E-03&   1.93&     0.23E-03&   1.94&   -1.3271\\
& 640&     0.67E-05&   1.98&     0.46E-04&   1.97&     0.58E-04&   1.98&   -1.4150\\
&1280&     0.17E-05&   1.99&     0.12E-04&   1.99&     0.15E-04&   2.00&   -1.4844\\
\hline
 \multirow{4}{1em}{$ P^2$} 
& 160&     0.63E-06&  -&     0.52E-05&  -&     0.87E-05&  -&   -2.6331\\
& 320&     0.88E-07&   2.85&     0.71E-06&   2.88&     0.11E-05&   2.95&   -2.9043\\
& 640&     0.11E-07&   2.95&     0.91E-07&   2.97&     0.14E-06&   3.00&   -3.0704\\
&1280&     0.14E-08&   2.99&     0.11E-07&   2.99&     0.17E-07&   3.01&   -3.1709\\
\hline
 \multirow{4}{1em}{$ P^3$} 
& 320&     0.64E-10&   -&     0.49E-09&   -&     0.72E-09&   -&   -\\
& 640&     0.45E-11&   3.82&     0.35E-10&   3.80&     0.52E-10&   3.79&   -4.3216\\
&1280&     0.29E-12&   3.93&     0.23E-11&   3.91&     0.34E-11&   3.92&   -4.6376\\
&2560&     0.19E-13&   3.97&     0.15E-12&   3.96&     0.22E-12&   3.96&   -4.8039\\
	\hline
	\end{tabular}
\end{table}

	\begin{table}[!h]
	\centering
	\small
	\caption{Example \ref{exa:globalp1}. Error of global projection $\pst u - u.$ Flux parameters (Case 1.5): $\alpha_1=0, \beta_1=\frac{1}{2h}$, $ \beta_2= h$.}%
	\label{tab:case15}
	\begin{tabular}{|c|c|c|c|c|c|c|cH|}
	\hline
	& N   & $L^1$ error & order&$L^2$ error & order& $L^{\infty}$ error & order & $\frac{\g}{\la} $ \\
\hline
 \multirow{4}{1em}{$ P^1$} 
& 320&     0.11E-03&   -&     0.63E-03&   -&     0.38E-03&   -&    2.0000\\
& 640&     0.28E-04&   2.00&     0.16E-03&   2.00&     0.95E-04&   2.00&    2.0000\\
&1280&     0.70E-05&   2.00&     0.39E-04&   2.00&     0.24E-04&   2.00&    2.0000\\
&2560&     0.18E-05&   2.00&     0.98E-05&   2.00&     0.60E-05&   2.00&    2.0000\\
\hline
 \multirow{4}{1em}{$ P^2$} 
& 320&     0.11E-06&   -&     0.71E-06&   -&     0.62E-06&   -&   -6.5000\\
& 640&     0.14E-07&   3.00&     0.89E-07&   3.00&     0.77E-07&   3.00&   -6.5000\\
&1280&     0.18E-08&   3.00&     0.11E-07&   3.00&     0.96E-08&   3.00&   -6.5000\\
&2560&     0.22E-09&   3.00&     0.14E-08&   3.00&     0.12E-08&   3.00&   -6.5000\\
	\hline
\multirow{4}{1em}{$ P^3$} 
& 320&     0.38E-10&   -&     0.25E-09&   -&     0.22E-09&   -&  -39.3333\\
& 640&     0.24E-11&   4.00&     0.16E-10&   4.00&     0.14E-10&   4.00&  -39.3333\\
&1280&     0.15E-12&   4.00&     0.99E-12&   4.00&     0.86E-12&   4.00&  -39.3333\\
&2560&     0.92E-14&   4.00&     0.62E-13&   4.00&     0.54E-13&   3.99&  -39.3333\\
	\hline
	\end{tabular}
\end{table}

	\begin{table}[!h]
	\centering
	\small
	\caption{Example \ref{exa:globalp1}. Error of global projection $\pst u - u.$ Flux parameters (Case 1.6.1): $\alpha_1=0.25,  \tilde \beta_1=\frac{k(k-1)}{2} + \frac{k(k+1)}{8}, \tilde \beta_2=1.0, A_1 = -1, A_2 = 2,3$. Note here $\lim_{h\to0} \lo, \lt = (-1)^k.$}
	\label{tab:case11}
	\begin{tabular}{|c|c|c|c|c|c|c|cH|}
	\hline
	& N   & $L^1$ error & order&$L^2$ error & order& $L^{\infty}$ error & order & $\frac{\g}{\la} $ \\
	\hline
  \multirow{4}{5em}{$ P^1$ \\ $A_2 = 2$ \\  $ \tilde \bo = \frac{1}{4}$}
& 640&     0.75E-05&   -&     0.52E-04&   -&     0.66E-04&   -&   -1.0265\\
&1280&     0.19E-05&   1.97&     0.13E-04&   1.97&     0.17E-04&   1.97&   -1.0132\\
&2560&     0.48E-06&   1.99&     0.34E-05&   1.98&     0.42E-05&   1.99&   -1.0066\\
&5120&     0.12E-06&   1.99&     0.84E-06&   1.99&     0.11E-05&   1.99&   -1.0033\\\hline
 \multirow{4}{5em}{$ P^2$ \\ $A_2 = 2$ \\  $\tilde \bo = \frac{7}{4}$}
& 640&     0.15E-06&   -	&     0.12E-05&   	&     0.23E-05&   -&   -1.1297\\
&1280&     0.39E-07&   1.94&     0.32E-06&   1.93&     0.61E-06&   1.94&   -1.0617\\
&2560&     0.98E-08&   1.97&     0.82E-07&   1.97&     0.16E-06&   1.97&   -1.0301\\
&5120&     0.25E-08&   1.98&     0.21E-07&   1.98&     0.39E-07&   1.99&   -1.0149\\
\hline
 \multirow{4}{5em}{$ P^2$ \\ $A_2 = 3$ \\ $\tilde \bo = \frac{7}{4}$}
& 640&     0.14E-04&  	 -&     0.12E-03&   -	&     0.21E-03&   	-&   -1.0012\\
&1280&     0.71E-05&   1.00&     0.58E-04&   1.00&     0.11E-03&   1.00&   -1.0003\\
&2560&     0.35E-05&   1.00&     0.29E-04&   1.00&     0.54E-04&   1.00&   -1.0001\\
&5120&     0.18E-05&   1.00&     0.15E-04&   1.00&     0.27E-04&   1.00&   -1.0000\\
\hline
 \multirow{4}{5em}{$ P^3$ \\ $A_2 = 2$ \\  $\tilde \bo = \frac{9}{2}$}
& 320&     0.12E-09&  -&     0.95E-09&   -&     0.20E-08&   -&   -2.1883\\
& 640&     0.78E-11&   3.99&     0.60E-10&   3.99&     0.13E-09&   3.99&   -1.4110\\
&1280&     0.49E-12&   3.99&     0.38E-11&   3.99&     0.80E-11&   3.99&   -1.1781\\
&2560&     0.31E-13&   4.00&     0.24E-12&   3.99&     0.51E-12&   3.97&   -1.0835\\
	\hline
	\end{tabular}
\end{table}

	\begin{table}[!h]
	\centering
	\small
	\caption{Example \ref{exa:globalp1}. Error of global projection $\pst u - u.$ Flux parameters (Case 1.7.2): $\alpha_1=0.25,  \tilde \beta_1=1, \tilde \beta_2=\frac{1}{2k(k-1)}, A_1 = -2,-3, A_2 = 1$. Note here $\lim_{h\to0} \lo, \lt = (-1)^{k+1}.$} %
	\label{tab:case12}
	\begin{tabular}{|c|c|c|c|c|c|c|cH|}
	\hline
	& N   & $L^1$ error & order&$L^2$ error & order& $L^{\infty}$ error & order & $\frac{\g}{\la} $ \\
\hline
 \multirow{4}{4em}{$ P^2$ \\ $A_1 = -3$ \\ $\tilde \bt = \frac{1}{4}$ } 
& 320&     0.28E-07&   -	&     0.21E-06&   -	&     0.24E-06&   	-&    1.0001\\
& 640&     0.35E-08&   3.00&     0.27E-07&   3.00&     0.31E-07&   3.00&    1.0000\\
&1280&     0.44E-09&   3.00&     0.33E-08&   3.00&     0.38E-08&   3.00&    1.0000\\
&2560&     0.55E-10&   3.00&     0.41E-09&   3.00&     0.48E-09&   3.00&    1.0000\\
\hline
 \multirow{4}{4em}{$ P^3$\\ $A_1 = -2$ \\ $\tilde \bt = \frac{1}{12}$ } 
& 320&     0.70E-08&   	-&     0.57E-07&   	-&     0.12E-06&   	-&    1.0308\\
& 640&     0.94E-09&   2.90&     0.77E-08&   2.90&     0.16E-07&   2.91&    1.0151\\
&1280&     0.12E-09&   2.95&     0.99E-09&   2.95&     0.20E-08&   2.95&    1.0074\\
&2560&     0.15E-10&   2.98&     0.13E-09&   2.98&     0.26E-09&   2.98&    1.0037\\
	\hline
\multirow{4}{4em}{$ P^3$\\ $A_1 = -3$ \\ $\tilde \bt = \frac{1}{12}$ } 
& 320&     0.16E-06&   -	&     0.13E-05&   	-&     0.24E-05&   	-&    1.0006\\
& 640&     0.40E-07&   2.00&     0.32E-06&   2.00&     0.61E-06&   2.00&    1.0001\\
&1280&     0.10E-07&   2.00&     0.79E-07&   2.00&     0.15E-06&   2.00&    1.0000\\
&2560&     0.25E-08&   2.00&     0.20E-07&   2.00&     0.38E-07&   2.00&    1.0000\\
	\hline
	\end{tabular}
\end{table}

\begin{table}[!h]
	\centering
	\caption{Example \ref{exa:globalp1}. Error of global projection $\pst u - u.$ Flux parameters (Case 1.7.2): $\alpha_1=0.25,  \tilde \beta_1=-1, \tilde \beta_2=\frac{1}{2k(k+1)}, A_1 = -2,-3, A_2 = 1$. Note that $\lim_{h\to0} \lo, \lt = (-1)^{k}=1.$} %
	\label{tab:case13}
	\begin{tabular}{|c|c|c|c|c|c|c|cH|}
	\hline
	& N   & $L^1$ error & order&$L^2$ error & order& $L^{\infty}$ error & order & $\frac{\g}{\la} $ \\
\hline
 \multirow{4}{4em}{  $ P^2$ \\ $A_1 = -2$ \\ $\tilde \bt = \frac{1}{12}$ } 
& 320&     0.72E-07&   2.99&     0.56E-06&   2.98&     0.94E-06&   2.97&   -1.0423\\
& 640&     0.90E-08&   2.99&     0.71E-07&   2.99&     0.12E-06&   2.99&   -1.0216\\
&1280&     0.11E-08&   3.00&     0.89E-08&   3.00&     0.15E-07&   2.99&   -1.0109\\
&2560&     0.14E-09&   3.00&     0.11E-08&   3.00&     0.19E-08&   3.00&   -1.0055\\
\hline
 \multirow{4}{4em}{ $ P^2$ \\ $A_1 = -3$ \\ $\tilde \bt = \frac{1}{12}$ } 
& 320&     0.80E-06&   2.01&     0.63E-05&   2.01&     0.12E-04&   2.01&   -1.0009\\
& 640&     0.20E-06&   2.00&     0.16E-05&   2.00&     0.30E-05&   2.00&   -1.0002\\
&1280&     0.50E-07&   2.00&     0.39E-06&   2.00&     0.75E-06&   2.00&   -1.0001\\
&2560&     0.13E-07&   2.00&     0.98E-07&   2.00&     0.19E-06&   2.00&   -1.0000\\
	\hline
	\end{tabular}
\end{table}

\begin{Example}
\label{exa:globalp2}
In this example, we consider global projection when the parameter choices are central-like fluxes belonging to Cases 1 and 2, for  smooth function  $u= e^{\cos(x)}$ on $[0, 2\pi]$ with a uniform mesh of size $h=2\pi/N$ and $k=1,2,3.$
\end{Example}

For central flux $(\ao,\bo,\bt) = (0,0,0),$ $\Gamma=-\frac{k^2}{2h},\Lambda=\frac{k}{2h}.$   If $k>1,$ $\frac{|\Gamma|}{|\Lambda|}=k>1,$ it belongs to Case 1, and if $k=1$, $\Gamma=-\Lambda$ and it belongs to Case 2. We conclude that $\pst$ exists and is unique for $k=1$ when $N$ is odd and $k >1$ for arbitrary $N.$ $\pst$ has optimal error estimates as proved in Lemmas \ref{lem:globalpcase1} and \ref{lem:globalpcase2optimal}. Our numerical test in Table \ref{tab:central} demonstrates optimal convergence rate for all $k.$

A similar flux is $(\ao,\bo,\bt) = (0,0,1).$ Lemma \ref{lem:globalpcase2optimal} yields first order convergence rate when $k=1$ as discussed in Remark \ref{rem:case2}. When $k=2,3$, similar to central flux, this parameter choice belongs to Case 1, showing optimal convergence rate. The numerical test in Table \ref{tab:central2} verifies the theoretical results.

\begin{table}[!h]
	\centering
	\small
	\caption{Example \ref{exa:globalp2}. Error of global projection $\pst u - u$. (Central flux) Flux parameters: $\alpha_1=0,  \beta_1=0 , \beta_2=0.$}
	\label{tab:central}
	\begin{tabular}{|c|c|c|c|c|c|c|cH|}
	\hline
	& N   & $L^1$ error & order&$L^2$ error & order& $L^{\infty}$ error & order &\\
	\hline
  \multirow{4}{1em}{$ P^1$} 
&  93&     0.12E-03&   -&     0.74E-03&   -&     0.55E-03&   -&   -1.0000\\
& 279&     0.13E-04&   2.00&     0.82E-04&   2.00&     0.61E-04&   2.00&   -1.0000\\
& 837&     0.15E-05&   2.00&     0.91E-05&   2.00&     0.68E-05&   2.00&   -1.0000\\
&2511&     0.17E-06&   2.00&     0.10E-05&   2.00&     0.76E-06&   2.00&   -1.0000\\
\hline
 \multirow{4}{1em}{$ P^2$} 
& 160&     0.11E-05&   -&     0.85E-05&   -&     0.10E-04&   -&   -2.0000\\
& 320&     0.14E-06&   3.00&     0.11E-05&   3.00&     0.13E-05&   2.99&   -2.0000\\
& 640&     0.17E-07&   3.00&     0.13E-06&   3.00&     0.16E-06&   3.00&   -2.0000\\
&1280&     0.22E-08&   3.00&     0.17E-07&   3.00&     0.20E-07&   3.00&   -2.0000\\
\hline
 \multirow{4}{1em}{$ P^3$} 
& 160&     0.11E-08&   -&     0.83E-08&   -&     0.11E-07&   -&   -3.0000\\
& 320&     0.68E-10&   4.00&     0.52E-09&   4.00&     0.68E-09&   4.00&   -3.0000\\
& 640&     0.42E-11&   4.00&     0.32E-10&   4.00&     0.42E-10&   4.00&   -3.0000\\
&1280&     0.27E-12&   4.00&     0.20E-11&   4.00&     0.26E-11&   4.00&   -3.0000\\
	\hline
	\end{tabular}
\end{table}

\begin{table}[!h]
	\centering
	\small
	\caption{Example \ref{exa:globalp2}. Error of global projection $\pst u - u$.   Flux parameters: $\alpha_1=0,  \beta_1=0 , \beta_2=1.$  }
	\label{tab:central2}
	\begin{tabular}{|c|c|c|c|c|c|c|cH|}
	\hline
	& N   & $L^1$ error & order&$L^2$ error & order& $L^{\infty}$ error & order &\\
	\hline
  \multirow{4}{1em}{$ P^1$} 
&  93&     0.21E-01&   -&     0.12E+00&   -&     0.68E-01&   -&   -1.0000\\
& 279&     0.72E-02&   1.00&     0.40E-01&   1.00&     0.23E-01&   1.00&   -1.0000\\
& 837&     0.24E-02&   1.00&     0.13E-01&   1.00&     0.75E-02&   1.00&   -1.0000\\
&2511&     0.80E-03&   1.00&     0.44E-02&   1.00&     0.25E-02&   1.00&   -1.0000\\
\hline
 \multirow{4}{1em}{$ P^2$} 
& 160&     0.11E-05&  -&     0.86E-05&  -&     0.10E-04&  -&  303.5775\\
& 320&     0.14E-06&   3.00&     0.11E-05&   3.00&     0.13E-05&   3.00&  609.1550\\
& 640&     0.17E-07&   3.00&     0.13E-06&   3.00&     0.16E-06&   3.00& 1220.3100\\
&1280&     0.22E-08&   3.00&     0.17E-07&   3.00&     0.20E-07&   3.00& 2442.6199\\
&2560&     0.27E-09&   3.00&     0.21E-08&   3.00&     0.25E-08&   3.00& 4887.2399\\
\hline
 \multirow{4}{1em}{$ P^3$} 
& 160&     0.27E-08&  -&     0.23E-07&  -&     0.36E-07&  -& 1219.3100\\
& 320&     0.17E-09&   4.00&     0.14E-08&   4.00&     0.22E-08&   4.00& 2441.6199\\
& 640&     0.11E-10&   4.00&     0.89E-10&   4.00&     0.14E-09&   4.00& 4886.2399\\
&1280&     0.66E-12&   4.00&     0.55E-11&   4.00&     0.87E-11&   4.00& 9775.4797\\
	\hline
	\end{tabular}
\end{table}

\begin{Example}
\label{exa:globalp3}
In this example, we consider global projection when the parameter choices belong to Case 3 for the smooth function $u=e^{\cos(x)}$ on $[0,2\pi]$ with uniform mesh size $h=2\pi/N$ and $k=1,2,3$. 
\end{Example}
 An example of Case 3 is shown in Table \ref{tab:case3}, where the parameters  are $(\ao,\tilde \bo, \tilde \bt) =( 0.25,-1, \frac{1}{2k(k-1)})$, $A_1 = -2,-3, A_2 = 1$, similar to the parameters in Table \ref{tab:case12}. The asymptotic behavior of $\lo, \lt$ when $h$ approaches $0$ is indeed similar to Table \ref{tab:case12}, that is, $\abs{\lo,\lt} = 1 + O({h^{-(A_1+1)/2}})$ and $\lim_{h\to 0} \lo, \lt =(-1)^{k+1}$. Same as previous examples, order reductions are only observed when $\lim_{h \to 0} \lo, \lt=1$, that is for $k=3$.
 
 We use this example to compare the error bounds obtained in   Lemmas \ref{lem:globalpcase3} and \ref{lem:globalpcase3optimal}. When $A_1 = -2$, $\delta = -(A_1+1) = 1$, we can verify $\abs{1-\lo^N} \sim O(1)$, i.e., $\delta' = 0$, thus by Lemma \ref{lem:globalpcase3optimal}, the convergence rate of $\pst$ is $k$, which agrees with the simulation and is better than the one in Lemma \ref{lem:globalpcase3} by half order.  When $A_1 = -3$, $\delta = -(A_1+1) = 2, \delta' = 0$, Lemma \ref{lem:globalpcase3} and Lemma \ref{lem:globalpcase3optimal} both show a convergence rate of $k-1$. These estimations are confirmed by the numerical results in Table \ref{tab:case3} when $k = 3$. 
 
We performed more numerical results of Case 3, and all are similar to those of Case 1 as long as the eigenvalues $\lo, \lt$ are approaching 1 at the same rate. Hence, we will not show more examples about Case 3. %

	\begin{table}[!h]
	\centering
	\small
	\caption{Example \ref{exa:globalp3}. Error of global projection $\pst u - u.$ Flux parameters (Case 3, and similar to Case 1.7.2 in Table \ref{tab:case12}): $\alpha_1=0.25,  \tilde \beta_1=-1$, $\tilde \beta_2=\frac{1}{2k(k-1)}, A_1 = -2,-3, A_2 = 1$. Note here $\lim_{h\to0} \lo, \lt = (-1)^{k+1}.$} %
	\label{tab:case3}
	\begin{tabular}{|c|c|c|c|c|c|c|cH|}
	\hline
	& N   & $L^1$ error & order&$L^2$ error & order& $L^{\infty}$ error & order & $\frac{\g}{\la} $ \\
\hline
 \multirow{4}{4em}{$ P^2$ \\ $A_1 = -3$ \\ $\tilde \bt = \frac{1}{4}$ } 
& 320&     0.28E-07&   -	&     0.21E-06&   -	&     0.24E-06&   	-&    1.0001\\
& 640&     0.35E-08&   3.00&     0.27E-07&   3.00&     0.31E-07&   3.00&    1.0000\\
&1280&     0.44E-09&   3.00&     0.33E-08&   3.00&     0.38E-08&   3.00&    1.0000\\
&2560&     0.55E-10&   3.00&     0.41E-09&   3.00&     0.48E-09&   3.00&    1.0000\\
\hline
 \multirow{4}{4em}{$ P^3$\\ $A_1 = -2$ \\ $\tilde \bt = \frac{1}{12}$ } 
& 320&     0.70E-08&   	-&     0.57E-07&   	-&     0.12E-06&   	-&    1.0308\\
& 640&     0.94E-09&   2.90&     0.77E-08&   2.90&     0.16E-07&   2.91&    1.0151\\
&1280&     0.12E-09&   2.95&     0.99E-09&   2.95&     0.20E-08&   2.95&    1.0074\\
&2560&     0.15E-10&   2.98&     0.13E-09&   2.98&     0.26E-09&   2.98&    1.0037\\
	\hline
\multirow{4}{4em}{$ P^3$\\ $A_1 = -3$ \\ $\tilde \bt = \frac{1}{12}$ } 
& 320&     0.16E-06&   -	&     0.13E-05&   	-&     0.24E-05&   	-&    1.0006\\
& 640&     0.40E-07&   2.00&     0.32E-06&   2.00&     0.61E-06&   2.00&    1.0001\\
&1280&     0.10E-07&   2.00&     0.79E-07&   2.00&     0.15E-06&   2.00&    1.0000\\
&2560&     0.25E-08&   2.00&     0.20E-07&   2.00&     0.38E-07&   2.00&    1.0000\\
	\hline
	\end{tabular}
\end{table}

\subsection{Numerical results of the DG scheme}
\label{sec:numerical2}

In this subsection, we show the numerical results of the DG scheme applied to the NLS equation.
For the time discretization, we use third order IMEX Runge-Kutta method \cite{ascher1997implicit} and fix $\Delta t= 1/10000,$ which is  small enough to guarantee that the spatial errors dominate. To be more precise, we treat the DG discretization of linear term $u_{xx}$ implicitly and nonlinear term $f(|u|^2)u$ explicitly.

\begin{Example} 
\label{exa:energy}
In this example, we verify the energy conservation property of our scheme by considering the following linear equation
$$
iu_t + u_{xx} = 0,
$$
with the progressive plane wave solution: $u(x,t) = Aexp(i(x-t)),$ with $A=1$.
\end{Example}

We use   $L^2$ projection as the numerical initial condition. In the discussion of stability condition, we derive that when $\mathrm{Im} \bt \geq 0, \mathrm{Im} \bt \leq 0, |\ao+\overline{\at}|^2 \leq -4 \mathrm{Im} \bt \mathrm{Im} \bt$, our scheme for \sch equation is stable. Furthermore, when $\ao + \at = 0, \bo,\bt$ are real numbers, the scheme is energy conservative. In this example, we compare two different parameter choices to verify the energy conservation property. The parameter choices are $(\ao, \at, \bo, \bt) = (0.25, -0.25, 1-i, 1+i)$, and $(\ao, \at, \bo, \bt) = (0.25, -0.25, 1,1)$  when $k=2, N = 40$, ending time $T=100$. Both are numerically stable flux parameters. For the first set of parameters, we expect energy decay due to the contributions from the imaginary part of $\bo, \bt$ as in \eqref{eqn:stab3}. For the second set of parameter, energy should be conserved.

In Fig. \ref{fig:tdiff}, we verify that as $t$ increases from $0$ to $100$, the flux with only real parameters preserve $\|u_h\|_{L^2(I)}$, while the flux with complex numbers have much larger errors. More precisely, for real parameters, $\|u_h(0,\cdot)\|_{L^2(I)} - \|u_h(100,\cdot)\|_{L^2(I)} = 7.9E$-09, for complex parameters, $\|u_h(0,\cdot)\|_{L^2(I)} -\| u_h(100,\cdot)\|_{L^2(I)} = 5.7E$-04.  

\begin{figure}[!ht]
	\centering
	\caption{ Example \ref{exa:energy}. Absolute difference of $\|u_h(t, \cdot)\|_{L^2(I)}$ with $\|u_h(0, \cdot)\|_{L^2(I)}$ with two sets of parameters $(\ao, \at, \bo, \bt) = (0.25, -0.25, 1-i, 1+i)$ (denoted by ``imag") and $(\ao, \at, \bo, \bt) = (0.25, -0.25, 1,1)$ (denoted by ``real")   when $k=2, N = 40$, ending time $T_e=100$. }
	\label{fig:tdiff}
	\includegraphics[width = 0.7 \textwidth]{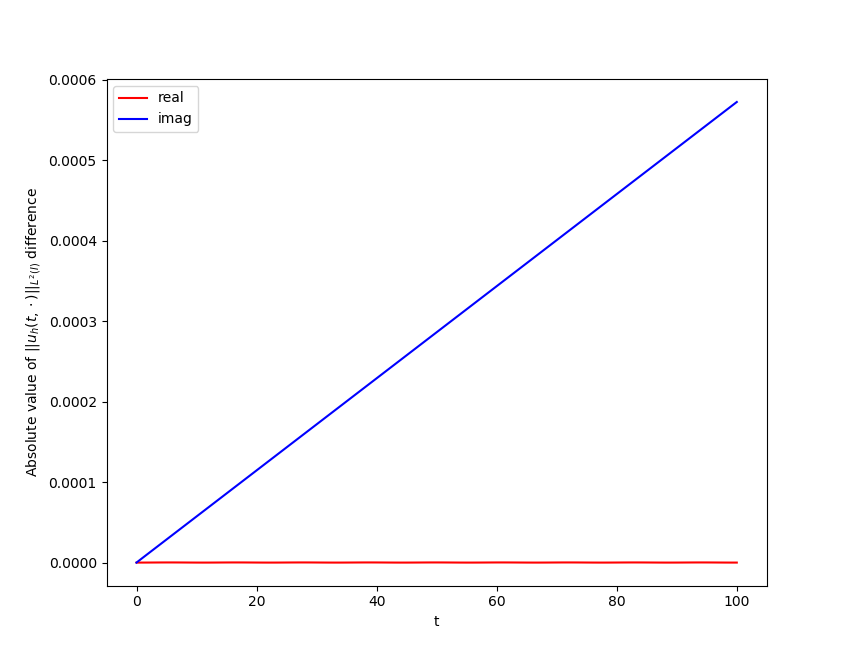}
\end{figure}

\begin{Example}
\label{exa:numerical1}
Accuracy test for NLS equation
\begin{equation}
\label{eqn:nlswave}
iu_t+u_{xx}+ |u|^2u+ |u|^4u=0,
\end{equation}
which admits a progressive plane wave solution: $u(x,t)=A\textrm{exp}(i(cx-\omega t)),$ where $\omega=c^2
-|A|^2 - |A|^4$ with $c=1, A=1$.  
\end{Example}
  For numerical initial condition, $\pst$ is used when applicable, otherwise standard $L^2$ projection is applied. We use six sets of parameters. The numerical errors and orders are shown in Tables \ref{tab:numcentral}, \ref{tab:numcentral2}, \ref{tab:numlocal1}, \ref{tab:numlocal2}, \ref{tab:numglobal} and \ref{tab:numglobal2}, where corresponding projection results are listed in Tables \ref{tab:central}, \ref{tab:central2}, \ref{tab:localp}, \ref{tab:localp2} , \ref{tab:case11} and \ref{tab:case13} respectively. Our numerical experiments show that the order of convergence for the scheme is the same as the order of error estimates for the projection $\pst$. 
  
  We would like to make some additional comments on Tables \ref{tab:numcentral} and \ref{tab:numcentral2}, whose parameter choices belong to Case 2 when $k=1$. The existence of $\pst$ requires $N$ to be odd for this case. However, this assumption is not needed for the optimal convergence rate of the numerical scheme for \eqref{eqn:nlswave} as shown in Tables \ref{tab:numcentral} and \ref{tab:numcentral2}. Similar comments have been made in \cite{XingKdvDG}.

\begin{table}[!h]
	\centering
	\small
	\caption{Example \ref{exa:numerical1}. Error in $L^1$, $L^2$ and $L^\infty$ norm for solving NLS equation \eqref{eqn:nlswave} using central flux (corresponding to Case 2 in Table \ref{tab:central}) $\alpha_1 = \beta_1 = \beta_2 = 0$, ending time $T_e=1$.} %
	\label{tab:numcentral}
	\begin{tabular}{|c|c|c|c|c|c|c|c|}
	\hline
	& N   & $L^1$ error & order&$L^2$ error & order& $L^{\infty}$ error & order \\
\hline
 \multirow{5}{1em}{$ P^1$} 
& 40&     0.28E-02&   -&     0.22E-02&   -&     0.27E-02&   -\\
& 80&     0.71E-03&   2.00&     0.56E-03&   2.00&     0.67E-03&   2.02\\
& 160&     0.18E-03&   2.00&     0.14E-03&   2.00&     0.17E-03&   2.01\\
& 320&     0.45E-04&   2.00&     0.35E-04&   2.00&     0.41E-04&   2.00\\
& 640&     0.11E-04&   2.00&     0.88E-05&   2.00&     0.10E-04&   2.00\\
\hline
 \multirow{5}{1em}{$ P^2$} 
& 40&     0.13E-03&  -&     0.11E-03&  -&     0.16E-03&  -\\
&  80&     0.16E-04&   2.99&     0.14E-04&   2.99&     0.20E-04&   3.00\\
& 160&     0.21E-05&   3.00&     0.18E-05&   3.00&     0.25E-05&   3.01\\
& 320&     0.26E-06&   3.00&     0.22E-06&   3.00&     0.31E-06&   3.00\\
& 640&     0.32E-07&   3.00&     0.27E-07&   3.00&     0.39E-07&   3.00\\
	\hline
\multirow{5}{1em}{$ P^3$} 
&  40&     0.22E-06&  -&     0.18E-06&  -&     0.24E-06&  -\\
&  80&     0.16E-07&   3.76&     0.13E-07&   3.80&     0.13E-07&   4.16\\
& 160&     0.10E-08&   4.00&     0.79E-09&   4.00&     0.84E-09&   4.00\\
& 320&     0.62E-10&   4.00&     0.49E-10&   4.00&     0.52E-10&   4.00\\
& 640&     0.39E-11&   3.99&     0.31E-11&   3.99&     0.33E-11&   3.96\\
	\hline
	\end{tabular}
\end{table}

\begin{table}[!h]
	\centering
	\small
	\caption{Example \ref{exa:numerical1}. Error in $L^1$, $L^2$ and $L^\infty$ norm for solving NLS equation \eqref{eqn:nlswave}  using flux parameters (corresponding to  Case 2 in Table \ref{tab:central2}): $\alpha_1= \beta_1=0,  \beta_2=1$, ending time $T_e=1$.}   %
	\label{tab:numcentral2}
	\begin{tabular}{|c|c|c|c|c|c|c|c|}
	\hline
	& N   & $L^1$ error & order&$L^2$ error & order& $L^{\infty}$ error & order \\
	\hline
 \multirow{5}{1em}{$ P^1$}
& 40&     0.17E+00&   -&     0.13E+00&   -&     0.14E+00&   -\\
&  80&     0.92E-01&   0.90&     0.72E-01&   0.89&     0.75E-01&   0.87\\
& 160&     0.48E-01&   0.94&     0.38E-01&   0.94&     0.38E-01&   0.97\\
& 320&     0.24E-01&   0.97&     0.19E-01&   0.97&     0.19E-01&   0.98\\
& 640&     0.12E-01&   0.98&     0.97E-02&   0.98&     0.98E-02&   0.99\\
\hline
 \multirow{5}{1em}{$ P^2$}
& 40&     0.13E-03&  -&     0.11E-03&  -&     0.17E-03&  -\\
&  80&     0.16E-04&   3.00&     0.14E-04&   3.00&     0.20E-04&   3.02\\
& 160&     0.21E-05&   3.00&     0.18E-05&   3.00&     0.25E-05&   3.01\\
& 320&     0.26E-06&   3.00&     0.22E-06&   3.00&     0.31E-06&   3.01\\                         
& 640&     0.32E-07&   3.00&     0.27E-07&   3.00&     0.39E-07&   3.00\\  
\hline
 \multirow{5}{1em}{$ P^3$}
&  40&     0.68E-06&  -&     0.56E-06&  -&     0.83E-06&  -\\
&  80&     0.42E-07&   4.00&     0.35E-07&   4.01&     0.51E-07&   4.01\\
& 160&     0.26E-08&   4.00&     0.22E-08&   4.00&     0.32E-08&   4.00\\
& 320&     0.16E-09&   4.00&     0.14E-09&   4.00&     0.20E-09&   4.00\\
& 640&     0.10E-10&   4.00&     0.85E-11&   4.00&     0.13E-10&   4.00\\
	\hline
	\end{tabular}
\end{table}

\begin{table}[!h]
	\centering
	\small
	\caption{Example \ref{exa:numerical1}. Error in $L^1$, $L^2$ and $L^\infty$ norm for solving NLS equation \eqref{eqn:nlswave} using flux  parameters (corresponding to  Table \ref{tab:localp}) $\alpha_1 = 0.3$, $\beta_1 = \beta_2 = 0.4$, ending time $T_e=1$.} %
	\label{tab:numlocal1}
	\begin{tabular}{|c|c|c|c|c|c|c|c|}
	\hline
	& N   & $L^1$ error & order&$L^2$ error & order& $L^{\infty}$ error & order \\
\hline
 \multirow{5}{1em}{$ P^1$} 
&  40&     0.69E-01&   -&     0.54E-01&   -&     0.59E-01&   -\\
&  80&     0.37E-01&   0.89&     0.29E-01&   0.89&     0.30E-01&   0.95\\
& 160&     0.19E-01&   0.95&     0.15E-01&   0.95&     0.15E-01&   0.98\\
& 320&     0.98E-02&   0.98&     0.77E-02&   0.98&     0.78E-02&   0.99\\
& 640&     0.50E-02&   0.99&     0.39E-02&   0.99&     0.39E-02&   0.99\\
\hline
 \multirow{5}{1em}{$ P^2$} 
& 40&     0.14E-03&  -&     0.12E-03&  -&     0.18E-03&  -\\
&  80&     0.17E-04&   3.05&     0.15E-04&   3.06&     0.21E-04&   3.09\\
& 160&     0.21E-05&   3.03&     0.18E-05&   3.03&     0.26E-05&   3.05\\
& 320&     0.26E-06&   3.01&     0.22E-06&   3.01&     0.32E-06&   3.02\\
& 640&     0.32E-07&   3.01&     0.28E-07&   3.01&     0.40E-07&   3.01\\
	\hline
\multirow{5}{1em}{$ P^3$} 
& 40&     0.69E-06&  -&     0.57E-06&  -&     0.85E-06&  -\\
&  80&     0.42E-07&   4.02&     0.35E-07&   4.02&     0.52E-07&   4.03\\
& 160&     0.26E-08&   4.01&     0.22E-08&   4.01&     0.32E-08&   4.01\\
& 320&     0.16E-09&   4.00&     0.14E-09&   4.00&     0.20E-09&   4.01\\                            
& 640&     0.10E-10&   4.00&     0.85E-11&   4.00&     0.13E-10&   3.99\\
	\hline
	\end{tabular}
\end{table}

\begin{table}[!h]
	\centering
	\small
	\caption{Example \ref{exa:numerical1}. Error in $L^1$, $L^2$ and $L^\infty$ norm for solving NLS equation \eqref{eqn:nlswave} using flux  parameters (corresponding to  Table \ref{tab:localp2}) $\alpha_1 = 0.3$, $\beta_1 = 0.4h, \beta_2 = 0.4/h$, ending time $T_e=1$.} %
	\label{tab:numlocal2}
	\begin{tabular}{|c|c|c|c|c|c|c|c|}
	\hline
	& N   & $L^1$ error & order&$L^2$ error & order& $L^{\infty}$ error & order \\
\hline
 \multirow{5}{1em}{$ P^1$} 
&  40&     0.66E-02&   -&     0.57E-02&   -&     0.97E-02&   -\\
&  80&     0.24E-02&   1.42&     0.20E-02&   1.50&     0.33E-02&   1.56\\
& 160&     0.43E-03&   2.51&     0.35E-03&   2.56&     0.52E-03&   2.66\\                              
& 320&     0.11E-03&   2.00&     0.86E-04&   2.00&     0.13E-03&   1.99\\                                   
& 640&     0.27E-04&   2.00&     0.22E-04&   2.00&     0.33E-04&   1.99\\      
\hline
 \multirow{5}{1em}{$ P^2$} 
& 40&     0.36E-03&  -&     0.31E-03&  -&     0.56E-03&  -\\
&  80&     0.45E-04&   2.99&     0.39E-04&   2.99&     0.70E-04&   3.01\\
& 160&     0.56E-05&   3.00&     0.49E-05&   3.00&     0.87E-05&   3.00\\
& 320&     0.70E-06&   3.00&     0.62E-06&   3.00&     0.11E-05&   3.00\\
& 640&     0.88E-07&   3.00&     0.77E-07&   3.00&     0.13E-06&   3.00\\        
	\hline
\multirow{5}{1em}{$ P^3$} 
&  40&     0.79E-06&  -&     0.66E-06&  -&     0.11E-05&  -\\
&  80&     0.49E-07&   4.00&     0.41E-07&   4.00&     0.66E-07&   4.00\\
& 160&     0.31E-08&   4.00&     0.26E-08&   4.00&     0.41E-08&   4.00\\
& 320&     0.19E-09&   4.00&     0.16E-09&   4.00&     0.26E-09&   4.00\\
& 640&     0.12E-10&   4.00&     0.10E-10&   4.00&     0.16E-10&   4.00\\
	\hline
	\end{tabular}
\end{table}

\begin{table}[!h]
	\centering
	\small
	\caption{Example \ref{exa:numerical1}. Error in $L^1$, $L^2$ and $L^\infty$ norm for solving NLS equation \eqref{eqn:nlswave}   using flux parameters (corresponding to Case 1.6.1 in Table \ref{tab:case11}): $\alpha_1=0.25,  \tilde \beta_1=\frac{k(k-1)}{2} + \frac{k(k+1)}{8}, \tilde \beta_2=1.0, A_1 = -1, A_2 = 2,3$, ending time $T_e=1$. } %
	\label{tab:numglobal}
	\begin{tabular}{|c|c|c|c|c|c|c|c|}
	\hline
	& N   & $L^1$ error & order&$L^2$ error & order& $L^{\infty}$ error & order \\
	\hline
  \multirow{5}{5em}{$ P^1$ \\ $A_2 = 2$ \\  $ \tilde \bo = \frac{1}{4}$}
&  40&     0.41E-02&   -&     0.37E-02&   -&     0.72E-02&   -\\
&  80&     0.12E-02&   1.77&     0.10E-02&   1.82&     0.21E-02&   1.80\\
& 160&     0.31E-03&   1.93&     0.25E-03&   2.05&     0.39E-03&   2.39\\
& 320&     0.87E-04&   1.86&     0.69E-04&   1.87&     0.10E-03&   1.94\\
& 640&     0.23E-04&   1.93&     0.18E-04&   1.94&     0.26E-04&   1.97\\
\hline
 \multirow{5}{5em}{$ P^2$ \\ $A_2 = 2$ \\  $\tilde \bo = \frac{7}{4}$}
&  40&     0.49E-04&  -&     0.49E-04&  -&     0.13E-03&  -\\
&  80&     0.83E-05&   2.55&     0.73E-05&   2.74&     0.14E-04&   3.23\\
& 160&     0.31E-05&   1.44&     0.29E-05&   1.32&     0.65E-05&   1.12\\
& 320&     0.95E-06&   1.69&     0.92E-06&   1.69&     0.20E-05&   1.70\\
& 640&     0.26E-06&   1.85&     0.25E-06&   1.86&     0.55E-06&   1.87\\
\hline
 \multirow{5}{5em}{$ P^2$ \\ $A_2 = 3$ \\ $\tilde \bo = \frac{7}{4}$}
&  40&     0.36E-03&  -&     0.34E-03&  -&     0.74E-03&  -\\
&  80&     0.21E-03&   0.78&     0.20E-03&   0.76&     0.43E-03&   0.77\\
& 160&     0.11E-03&   0.92&     0.11E-03&   0.92&     0.23E-03&   0.92\\
& 320&     0.56E-04&   1.00&     0.53E-04&   1.00&     0.11E-03&   0.99\\
& 640&     0.28E-04&   1.00&     0.27E-04&   1.00&     0.58E-04&   1.00\\
\hline
 \multirow{4}{5em}{$ P^3$ \\ $A_2 = 2$ \\  $\tilde \bo = \frac{9}{2}$}
& 40&     0.19E-05&  -&     0.19E-05&  -&     0.43E-05&  -\\
&  80&     0.43E-07&   5.50&     0.38E-07&   5.65&     0.84E-07&   5.66\\
& 160&     0.15E-08&   4.88&     0.15E-08&   4.68&     0.26E-08&   5.00\\
& 320&     0.91E-10&   4.00&     0.90E-10&   4.02&     0.17E-09&   3.94\\
&  640&     0.58E-11&   3.96&     0.57E-11&   3.99&     0.11E-10&   3.98\\
	\hline
	\end{tabular}
\end{table}

\begin{table}[!h]
	\centering
	\small
	\caption{Example \ref{exa:numerical1}. Error in $L^1$, $L^2$ and $L^\infty$ norm for solving NLS equation \eqref{eqn:nlswave}   using flux parameters (corresponding to  Case 1.7.2 in Table \ref{tab:case13}): $\alpha_1=0.25,  \tilde \beta_1=-1, \tilde \beta_2=\frac{1}{2k(k+1)}, A_1 = -2,-3, A_2 = 1$, ending time $T_e=1$.}%
	\label{tab:numglobal2}
	\begin{tabular}{|c|c|c|c|c|c|c|c|}
	\hline
	& N   & $L^1$ error & order&$L^2$ error & order& $L^{\infty}$ error & order \\
	\hline
 \multirow{5}{5em}{$ P^2$ \\$A_1 = -2$ \\ $\tilde \bo = \frac{1}{12}$}
&  40&     0.60E-04&  -&     0.54E-04&  -&     0.95E-04&  -\\
&  80&     0.76E-05&   2.99&     0.68E-05&   2.98&     0.12E-04&   2.96\\
& 160&     0.96E-06&   3.00&     0.85E-06&   3.00&     0.15E-05&   2.99\\
& 320&     0.12E-06&   3.00&     0.11E-06&   3.00&     0.19E-06&   2.99\\                            
& 640&     0.15E-07&   3.00&     0.13E-07&   3.00&     0.24E-07&   3.00\\ 
\hline
 \multirow{5}{5em}{$ P^2$ \\ $A_1 = -3$ \\ $\tilde \bo = \frac{1}{12}$}
&  40&     0.95E-04&  -&     0.85E-04&  -&     0.15E-03&  -\\
&  80&     0.21E-04&   2.22&     0.18E-04&   2.20&     0.33E-04&   2.18\\
& 160&     0.49E-05&   2.08&     0.44E-05&   2.07&     0.79E-05&   2.06\\
& 320&     0.12E-05&   2.02&     0.11E-05&   2.02&     0.20E-05&   2.02\\
& 640&     0.29E-06&   2.02&     0.27E-06&   2.02&     0.48E-06&   2.02\\
	\hline
	\end{tabular}
\end{table}

\begin{Example} 
\label{exa:soliton}
A simulation for the NLS equation
\begin{equation}
\label{eqn:soliton1}
iu_t+u_{xx}+2|u|^2u=0
\end{equation}
with  double-soliton collision  
\begin{equation}
\label{eqn:doublesoliton}
u(x,t)=\sech(x+10-4t)\exp(i(2(x+10)-3t)) + \sech(x-10+4t)\exp(i(-2(x-10)-3t)).
\end{equation}
\end{Example}

We use periodic boundary condition and $L^2$ projection initialization to run the simulation for double-soliton collision solution. The two waves propagate in opposite directions and collide at $t=2.5$, after that, the two waves separate. Such behaviors are accurately captured by our numerical simulations, see Figure \ref{fig:soliton1} for details.

\begin{figure}[!ht]
\centering
\begin{tabular}{cc}
\includegraphics[width=.45\textwidth]{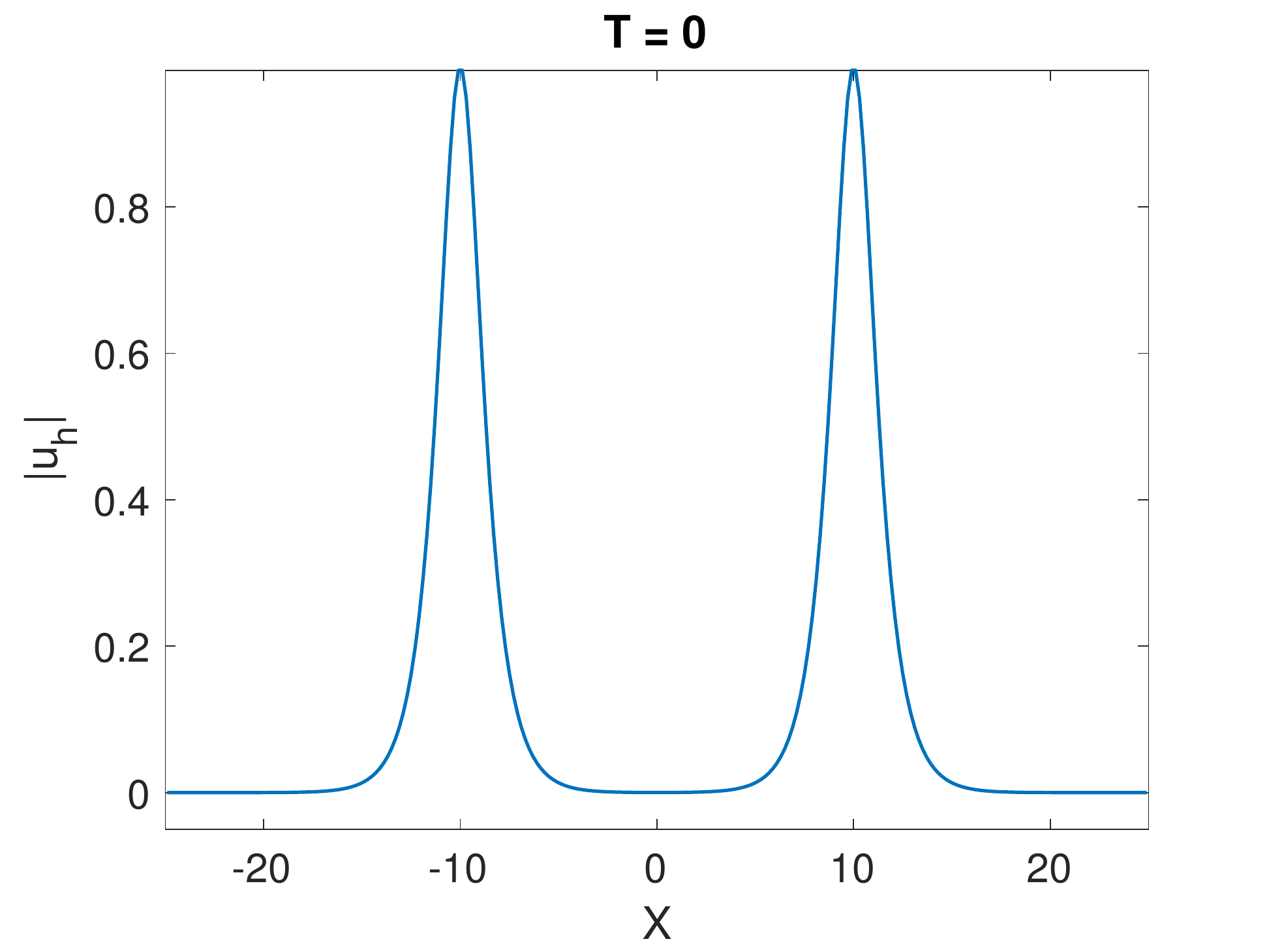}
\includegraphics[width=.45\textwidth]{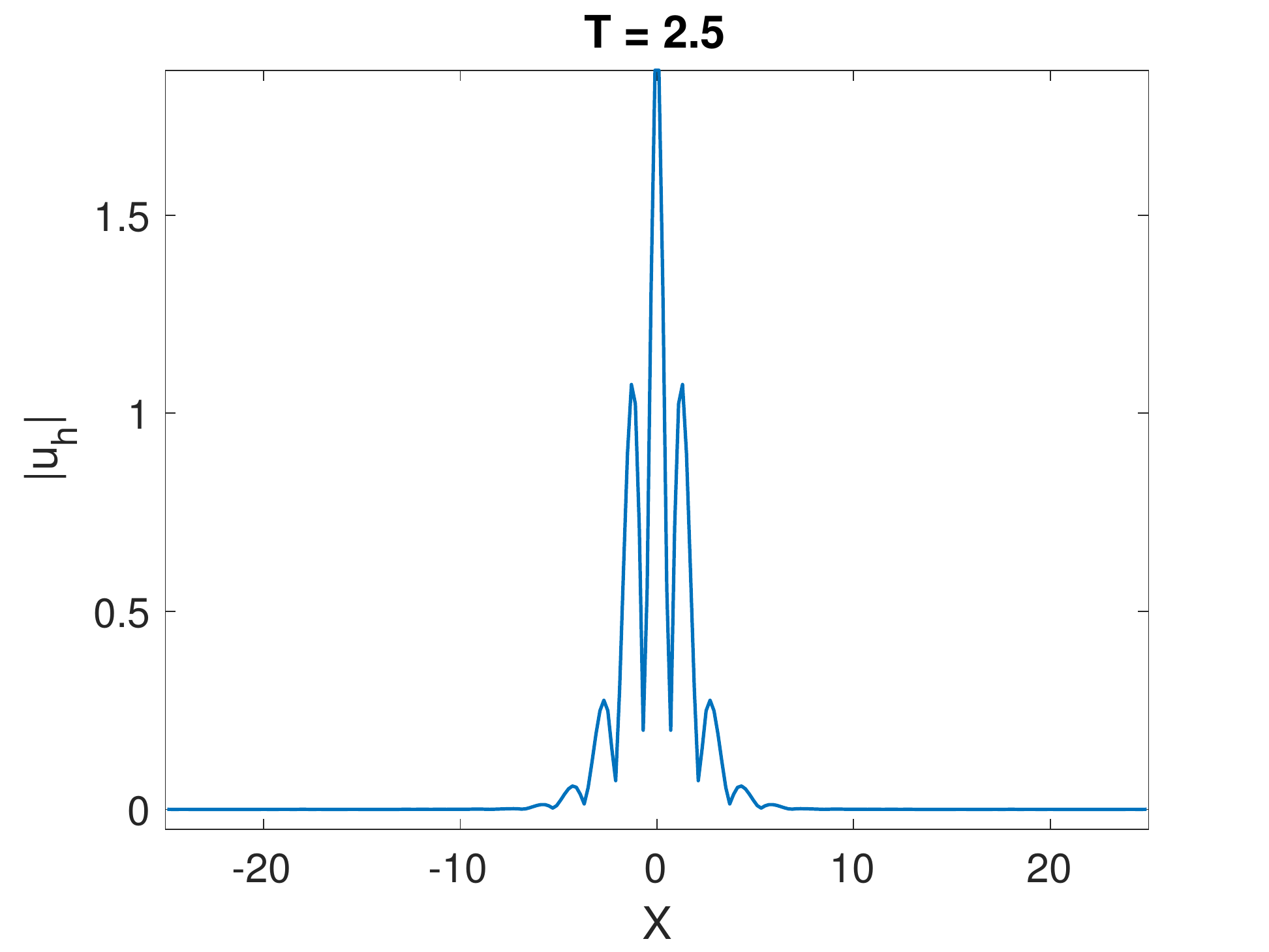}\\
\includegraphics[width=.45\textwidth]{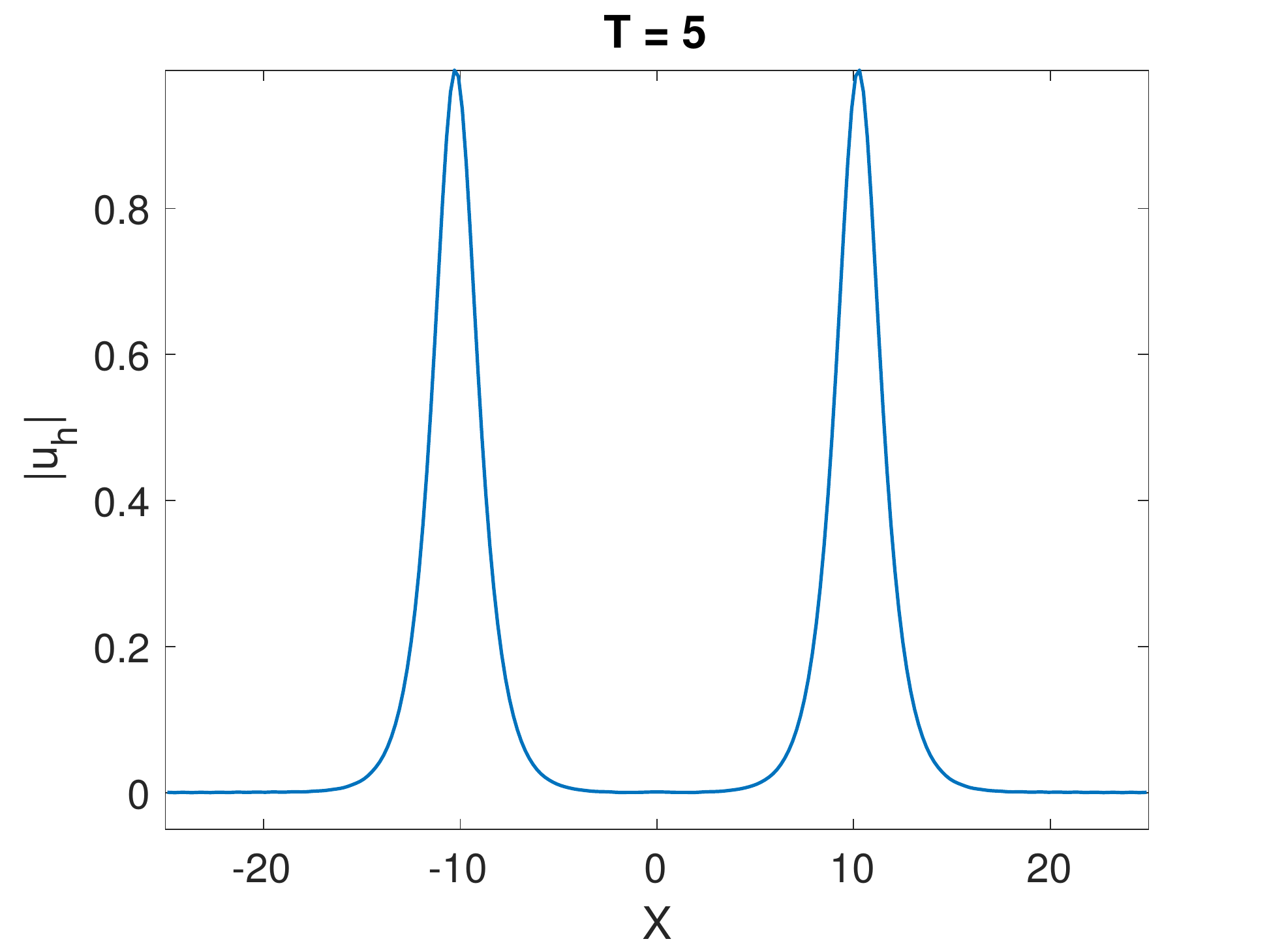}
\includegraphics[width=.45\textwidth]{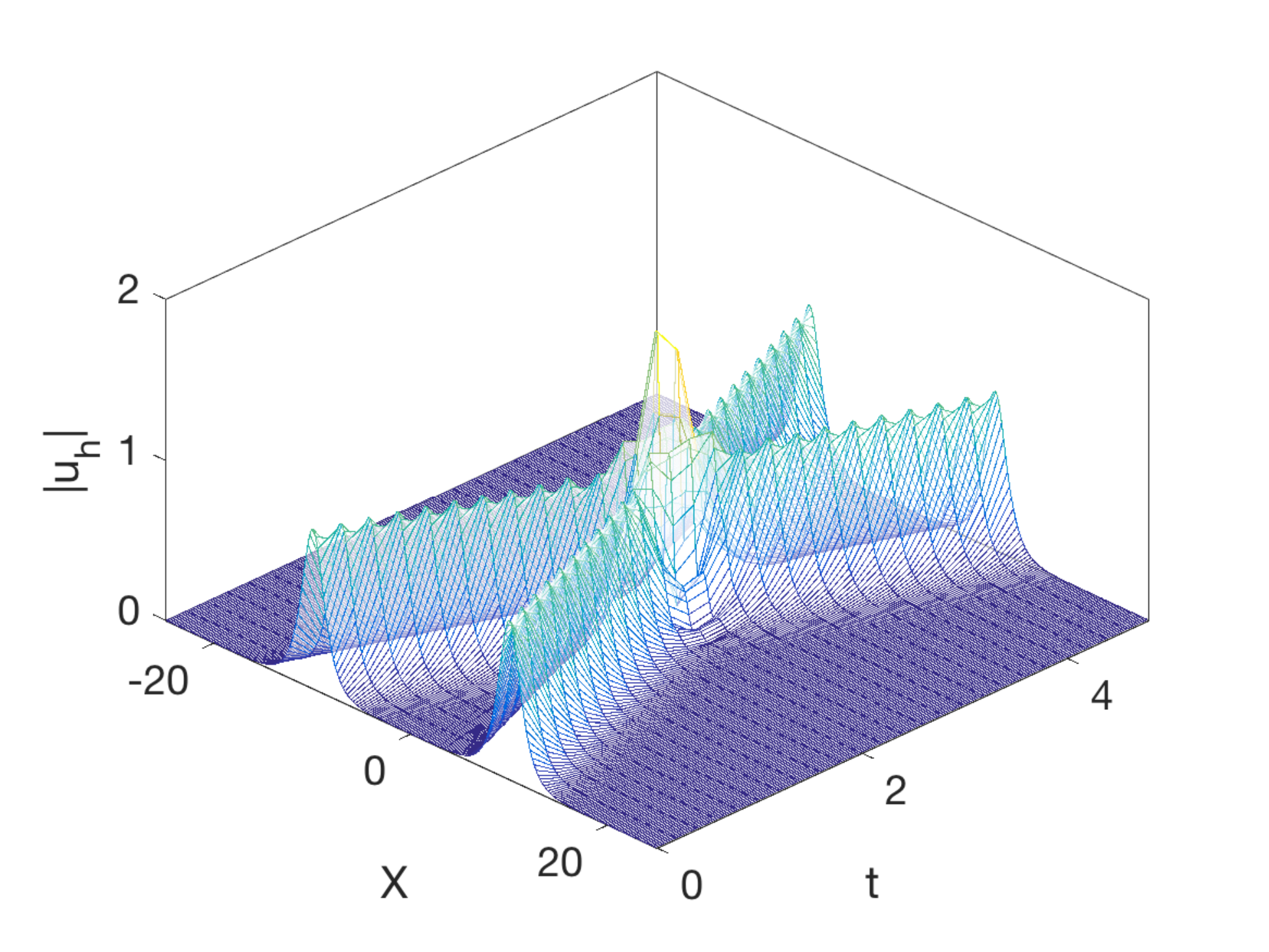}
\end{tabular}
\caption{Example \ref{exa:soliton}. Double soliton collision graphs at $t=0,2.5,5$ and a $x-t$ plot of the numerical solution. $N=250, P^2$ elements with periodic boundary conditions on [-25,25]. Central flux ($\alpha_1 = \beta_1 = \beta_2 = 0$) is used.}
\label{fig:soliton1}
\end{figure}

\section{Conclusions and future work}
\label{sec:conclusion}
In this paper, we studied the ultra-weak DG method with a general class of numerical fluxes  for solving one-dimensional nonlinear \sch equation with periodic boundary conditions. Semi-discrete $L^2$ stability and   error estimates are obtained  when the polynomial degree $k\geq 1$. Focusing on the real parameters, we performed detailed investigation of the associated projection operators. Our analysis assume the dependence of parameters on the mesh size $h$ can be freely enforced, hence several cases follow. A variety of analytic tools are employed, including decoupling of global projection into local projection, analysis of block-circulant matrix and Fourier analysis. We acquire error bounds that are sharp in most cases from numerical verifications. Future work includes improvement of the error bounds for some suboptimal cases, superconvergence studies and generalization to higher-dimensions.

\appendix
 
\section{Appendix}
\subsection{ Proof of Lemma \ref{lem:localp} }
\label{sec:a01}

First, we consider the case when $\bo \neq 0$. Define the difference operator $Wu= \pst u- \ppm u$, then \eqref{eqn:pst2} implies:
\beq
\label{eqn:W1}
\begin{aligned}
\intj W u  \, \vh dx &=0 \quad & &\forall \vh \in P^{k-2}(I_j), \\
Wu^+ + \frac{\ot+\ao}{\bo}(W u)_x^+ & =u - (\ppm u)^+ & & \textrm{at} \, x_{\jm},\\
Wu^- -\frac{\ot-\ao}{\bo}(W u)_x^- 	& = -\frac{\ot-\ao}{\bo}(u_x - (\ppm u)_x^-) & & \textrm{at} \, x_{\jp}.
\end{aligned}
\eeq

For $l\geq 0$, let $P_l(\xi)$ be the $l$-th order Legendre polynomials on [-1,1], with $\xi=\frac{2(x-x_j)}{h_j}$ on $I_j,$ and define $P_{j,l}(x)=P_l(\frac{2(x-x_j)}{h_j})=P_l(\xi)$. Then $W u$ can be expressed as:
\[
W u(x)=\sum_{l=0}^k a_{j,l}P_{j,l}(x)=\sum_{l=0}^k a_{j,l}P_l(\xi).
\]

By the first equation in \eqref{eqn:W1} and orthogonality of Legendre polynomials, one can get:
\[
a_{j,l}=0, \quad l=0,\cdots,k-2,\quad j=1,\cdots,N.
\]

We can then move on to solve for $a_{j,k-1}$ and $a_{j,k}$ on each cell directly by the second and third equations in \eqref{eqn:W1}. By properties of Legendre polynomials: $P_l(\pm1)=(\pm1)^l,\,P'_l(\pm1)=\ot(\pm1)^{l-1}l(l+1)$, the following $2\times2$ linear system holds on each cell $I_j$:

$$
\mathcal{M}_j \begin{bmatrix}
a_{j,k-1} \\
a_{j,k}
\end{bmatrix}
=\begin{bmatrix}
\phi_j	\\
\psi_j
\end{bmatrix},
$$
where
\beq
\notag
\mathcal{M}_j=\begin{bmatrix}
(\mathcal{M}_j)_{11}	&	(\mathcal{M}_j)_{12} \\
(\mathcal{M}_j)_{21}				&	(\mathcal{M}_j)_{22}
\end{bmatrix}
=
\begin{bmatrix}
(-1)^{k-1} + (-1)^{k} \frac{\ot+\ao}{\bo} \frac{k(k-1)}{h_j} 	&	(-1)^k + (-1)^{k-1}\frac{\ot+\ao}{\bo} \frac{k(k+1)}{h_j} \\
1  -\frac{\ot-\ao}{\bo}\frac{k(k-1)}{h_j}					&	1  -\frac{\ot-\ao}{\bo}\frac{k(k+1)}{h_j}
\end{bmatrix}
\eeq
and 
$\phi_j=(u - (\ppm u)^+)|_{x_{\jm}}$ and $\psi_j=-\frac{\ot-\ao}{\bo}(u_x - (\ppm u)_x^-)|_{x_{\jp}}.$

We can calculate the determinant of the matrix $\mathcal{M}_j$ to be $2(-1)^{k-1} + 2(-1)^{k}\frac{k^2}{\bo h_j} + 2(-1)^{k-1} \frac{\bt k^2(k^2-1)}{\bo h_j^2} = 2(-1)^{k-1} \Gamma_j / \bo.$ Hence, when $\Gamma_j  \ne 0, \, \forall j$, $\pst$ exists and is unique.  
  We now move on to estimate the $a_{j,k-1}, a_{j,k}.$ Clearly, 
\begin{eqnarray*}
a_{j,k-1}=\frac{1}{det\mathcal{M}_j} ( (\mathcal{M}_j)_{22} \phi_j -(\mathcal{M}_j)_{12} \psi_j ) \\
a_{j,k}=\frac{1}{det\mathcal{M}_j} (- (\mathcal{M}_j)_{21} \phi_j +(\mathcal{M}_j)_{11} \psi_j),
\end{eqnarray*}
and from the projection property of $\ppm$, $|\phi_j| \le Ch^{k+1}_j |u|_{W^{k+1,\infty}(I_j)}, |\psi_j| \le C  h^{k}_j  |\frac{\ot-\ao}{\bo}||u|_{W^{k+1,\infty}(I_j)}.$ The error estimates can be obtained based on the following cases.
\begin{itemize}
\item
If $k=1$, then 
\[
\begin{aligned}
a_{j,0}& 
		 = \frac{1}{2\Gamma_j} \big( (\bo -(\ot-\ao) \frac{2}{h_j}) (u - \ppm u)^+|_{x_{\jm}} - (\ot-\ao - \frac{2\bt}{h_j})(u_x - (\ppm u)_x)^-|_{x_{\jp}} \big),
		 \\
a_{j,1}&   = \frac{1}{2\Gamma_j} \big( -\bo (u - \ppm u)^+|_{x_{\jm}} -(\ot-\ao)(u_x - (\ppm u)_x)^-|_{x_{\jp}}\big).
\end{aligned}
\]

Thus we have estimates
\begin{eqnarray*}
|a_{j,0}|&\le&  \frac{Ch_j^2 |u|_{W^{2,\infty}(I_j)}}{|\Gamma_j|}  \max \left( \left | \bo - \frac{1-2\ao}{h_j} \right |, \left | \frac{\ot-\ao}{h_j} - \frac{2\bt}{h_j^2} \right | \right)  , \\
|a_{j,1}|&\le&  \frac{Ch_j^2 |u|_{W^{2,\infty}(I_j)}}{|\Gamma_j|} \max \left ( | \bo|, \left |\frac{ \ot - \ao}{h_j} \right | \right ) .
\end{eqnarray*}

Then,%
\begin{eqnarray}
\label{eqn:pf1}
&&\|\pst u-\ppm u\|_{L^\infty (I_j)} =\| a_{j,0}P_{0}(\xi) + \alpha_{j,1} P_1 (\xi) \|_{L^\infty(I_j)}  \notag \\
 &&  
 \leq  \frac{Ch_j^2 |u|_{W^{2,\infty}(I_j)}}{|\Gamma_j|} \max \left ( | \bo|, \left |\frac{ \ot - \ao}{h_j} \right |, \left | \frac{\bt}{h_j^2} \right | \right ).
\end{eqnarray}

Combining with the error estimates for $\ppm$ and the mesh regularity assumption, we get
\begin{eqnarray*}
&&\|\pst u-u\|_{L^p(I)} %
 \leq  Ch^2 |u|_{W^{2,\infty}(I)} \left ( 1+ \frac{ \max \left ( | \bo|,\frac{ |\ot - \ao|}{ h} ,  \frac{|\bt|}{h^2}  \right )}{\min_j{|\Gamma_j|}} \right )  , \quad p=2,\infty. 
\end{eqnarray*}

\item If $k>1$, then we need to discuss the case when $\bt=0$ or $\bt \ne 0.$

If $\bt =0,$ then $\ao=\pm \frac{1}{2}.$ %
When $\ao=\frac{1}{2},$ we have $\psi_j=0,$ and
\begin{eqnarray*}
|a_{j,k-1}|&\le&  C  h^{k+1}_j  |u|_{W^{k+1,\infty}(I_j)} \frac{|\bo|}{|\Gamma_j|},\\
|a_{j,k}|&\le&  C   h^{k+1}_j  |u|_{W^{k+1,\infty}(I_j)} \frac{|\bo|}{|\Gamma_j|}.
\end{eqnarray*}

Therefore,
\begin{eqnarray}
\label{eqn:superc1}
\|\pst u-\ppm u\|_{L^\infty(I_j)} \leq C h^{k+1}_j  |u|_{W^{k+1,\infty}(I_j)}\frac{|\bo|}{|\Gamma_j|}, 
\end{eqnarray}
implying a supercloseness between $\pst$ and $\ppm$ if $\bo/\Gamma_j=o(1)$.   In summary, we have %
\begin{eqnarray*}
&&
\|\pst u-u\|_{L^p(I)} \leq  C h^{k+1}  |u|_{W^{k+1,\infty}(I)} \left ( 1+   \frac{|\bo|}{\min_j|\Gamma_j|}  \right ) , \quad p=2,\infty. 
\end{eqnarray*}

When $\ao=-\frac{1}{2},$ we then should compare the projection with $\pt$ instead of $\ppm$. %
We skip the details of the calculations. The conclusion is similar, i.e.
\begin{eqnarray*}
&& 
\|\pst u-\pt u\|_{L^\infty(I_j)} \leq C h^{k+1}_j  |u|_{W^{k+1,\infty}(I_j)} \frac{|\bo|}{|\Gamma_j|} ,
\end{eqnarray*}
and
\begin{eqnarray*}
&&
\|\pst u-u\|_{L^p(I)} \leq  C h^{k+1}  |u|_{W^{k+1,\infty}(I)} \left ( 1+ \frac{|\bo|}{\min_j |\Gamma_j|} \right )  , \quad p=2,\infty. 
\end{eqnarray*}

If $\bt \ne 0,$ similar to previous case, we can show
\begin{eqnarray*}
|a_{j,0}|&\le&  \frac{Ch_j^{k+1} |u|_{W^{k+1,\infty}(I_j)}}{|\Gamma_j|}  \max \left( | \bo| , \left| \frac{\ot-\ao}{h_j} \right |, \left |  \frac{\bt}{h_j^2} \right | \right) , \\
|a_{j,1}|&\le& \frac{Ch_j^{k+1} |u|_{W^{k+1,\infty}(I_j)}}{|\Gamma_j|}  \max \left( | \bo|,  \left| \frac{\ot-\ao}{h_j} \right |, \left |  \frac{\bt}{h_j^2} \right | \right)  .
\end{eqnarray*}

Therefore, 
\begin{eqnarray*}
&&\|\pst u-u\|_{L^\infty(I_j)} \leq Ch_j^{k+1} |u|_{W^{k+1,\infty}(I_j)} \left ( 1+ \frac{\max \left( | \bo|,  \left| \frac{\ot-\ao}{h_j} \right |, \left |  \frac{\bt}{h_j^2} \right | \right)}{|\Gamma_j|} \right )
\end{eqnarray*}
and it leads to
\begin{equation*}
\|\pst u-u\|_{L^p(I)} \leq Ch^{k+1} |u|_{W^{k+1,\infty}(I)} \left ( 1+ \frac{\max \left( | \bo|,    \frac{|\ot-\ao|}{ h},    \frac{|\bt|}{h^2}  \right)}{\min_j |\Gamma_j|} \right ), \quad p = 2,\infty.
\end{equation*}
\end{itemize}

Finally, when $\bo =0, \bt \neq 0, \ao = \pm \ot$, we have the following estimates
\begin{eqnarray*}
&& \|\pst u-u\|_{L^p(I)} \leq  C h^{k+1}  |u|_{W^{k+1,\infty}(I)} \left ( 1+ \frac{|\bt| }{h^2\min_j |\Gamma_j|} \right )  , \quad p=2,\infty. 
\end{eqnarray*}

Summarizing all the estimates, we have shown \eqref{eqn:localpest} for all cases.

\subsection{ Proof of Lemma \ref{lem:globalp} }
\label{sec:a02}

 We adopt similar notations as in the proof of Lemma \ref{lem:localp}.
Define $\xi=\frac{2(x-x_j)}{h}$, and let
\[
\pst u(x)|_{I_j}=\sum_{l=0}^k \gamma_{j,l}P_{j,l}(x)=\sum_{l=0}^k \gamma_{j,l}P_l(\xi).
\]

By   \eqref{eqn:psts1} and orthogonality of Legendre polynomials, one can get:
\[
\gamma_{j,l}=\frac{2l+1}{2}\int_{-1}^{1} u(x_j+\frac{h}{2}\xi)P_l(\xi)d\xi, \quad l=0,\cdots,k-2,\quad j=1,\cdots,N.
\]

We can then move on to solve for $\gamma_{j,k-1}$ and $\gamma_{j,k}$ from \eqref{eqn:psts2}-\eqref{eqn:psts3}. At $x_{j+\ot}$, 
\beq
\label{eqn:pstcomp}
\begin{aligned}
\widehat{\pst u} & = \sum_{l=0}^{k}\{ \gamma_{j,l}((\ot +\ao)P_l(1)-\bt \toh P'_l(1)) \\
			& + \gamma_{j+1,l}((\ot -\ao) P_l(-1)+\bt \toh P'_l(-1)) \}\\
                     & = u(x_{\jp}),\\
\widetilde{\pst u_x} & = \sum_{l=0}^{k} \gamma_{j,l}((\ot -\ao)\toh P'_l(1)-\bo P_l(1)) \\
			& + \gamma_{j+1,l}((\ot +\ao)\toh P_l'(-1)+\bo P_l(-1)) \\
                           &=u_x(x_{\jp}).
\end{aligned}
\eeq
Combining \eqref{eqn:pstcomp} for all $j$  and using the periodic boundary condition will result in the following $2N\times2N$ linear system 
\begin{equation}
\label{eqn:matform}
M
\begin{bmatrix}
\gamma_{1, k-1} \\
\gamma_{1,k} \\
\cdots \\
\gamma_{N-1,k-1} \\
\gamma_{N-1,k} \\
\gamma_{N, k-1} \\
\gamma_{N,k}
\end{bmatrix}
=
\begin{bmatrix}
\phi_1 \\
\psi_1 \\
\cdots \\
\phi_N-1 \\
\psi_N-1 \\
\phi_{N} \\
\psi_{N} \\
\end{bmatrix}
\end{equation}
where  $M=circ(A,B,0_2,\cdots,0_2)$,  denoting a $2N\times 2N$ block-circulant matrix with first two rows as $(A,B,0_2,\cdots,0_2)$, with $0_2$ as a $2\times 2$ zero matrix, and
\begin{eqnarray}
A & = &
 \begin{bmatrix}
\ot+\ao & -\bt \\
-\bo	& \ot - \ao
\end{bmatrix}
\begin{bmatrix}
\pkmopos & \pkpos \\
\toh \dpkmopos & \toh \dpkpos
\end{bmatrix}, \label{eqn:A}\\
B & = &
\begin{bmatrix}
\ot-\ao & \bt \\
\bo	& \ot + \ao
\end{bmatrix}
\begin{bmatrix}
\pkmoneg & \pkneg \\
\toh \dpkmoneg & \toh \dpkneg
\end{bmatrix} , \label{eqn:B}
\end{eqnarray}
$\phi_{j}=u(x_{j+\ot}) - \sum_{l=0}^{k-2} \{\gamma_{j,l}((\ot +\ao)P_l(1)-\bt \toh P'_l(1)) + \gamma_{j+1,l}((\ot -\ao) P_l(-1)+\bt \toh P'_l(-1)) \}$, $\psi_{j}=u_x(x_{j+\ot}) - \sum_{l=0}^{k-2} \{ \gamma_{j,l}((\ot -\ao)\toh P'_l(1)-\bo P_l(1)) + \gamma_{j+1,l}((\ot +\ao)\toh P_l'(-1)+\bo P_l(-1)) \}$.
We can calculate that
\beq
\label{eqn:ABdet}
\begin{aligned}
detA=detB=\frac{-2k}{h}(\ao^2+\bo\bt-\frac{1}{4}) := \Lambda  \ne 0.
\end{aligned}
\eeq

It is clear that the existence and uniqueness of $\pst$ is equivalent to $detM\neq0.$ By a direct computation,  $detM=detA^N det(I_2-Q^N),$ where $I_2$ denotes the $2\times2$ identity matrix, and
\begin{align*}
Q &=-\inva B=\frac{(-1)^{k+1}}{\Lambda}
\begin{bmatrix}
c_1+c_2 & b_1+b_2 \\
b_1-b_2 & c_1-c_2
\end{bmatrix}, 
\end{align*}
with
\begin{align}
c_1 &=\bo +\frac{k^2(k^2-1)}{h^2}\bt -\frac{2k^2}{h}(\ao^2+\bo \bt+\frac{1}{4}):=\Gamma, \label{eqn:c1}\\
c_2 &=\frac{k}{h}(2\ao), \label{eqn:c2}\\
b_1 &=-\bo-\frac{k^2(k^2+1)}{h^2}\bt +\frac{2k^2}{h}(\ao^2+\bo\bt+\frac{1}{4}), \label{eqn:b1}\\
b_2 &=-\frac{2k^3}{h^2}\bt+\frac{2k}{h}(\ao^2+\bo\bt+\frac{1}{4}).
\label{eqn:b2}
\end{align}
 The eigenvalues of $Q$ are
\begin{equation}
\label{eqn:eig2}
\lambda_{1}=\frac{(-1)^{(k+1)}}{\Lambda}(\Gamma+ \sqrt{\Gamma^2-\Lambda^2}), \quad \lambda_{2}=\frac{(-1)^{(k+1)}}{\Lambda}(\Gamma- \sqrt{\Gamma^2-\Lambda^2}).
\end{equation}

Since $detQ=detB/detA= 1$, we have the relations $\lo \lt =1$ and 
\beq
\label{eqn:b1b2eqn}
b_1^2 - b_2^2 = \Gamma^2 - \Lambda^2 - c_2^2.
\eeq

Below we will discuss the existence and uniqueness of $\pst$  based on three cases depending on the relation of $\Gamma$ and $\Lambda.$

\underline{Case 1.} If $|\Gamma|>|\Lambda|$, then $\lambda_{1,2}$ are real and   different.  Therefore, we can perform eigenvalue decomposition of $Q$, 
$$
Q=T D T^{-1},
$$
where 
$$
D=\begin{bmatrix}
\lo  & 0 \\
0 & \lt
\end{bmatrix},
$$

and

\beq
\label{eqn:T}
T=
\begin{bmatrix}
1 & -\frac{b_1+b_2}{c_2+\sqrt{\Gamma^2-\Lambda^2}} \\
\frac{b_1-b_2}{c_2+\sqrt{\Gamma^2-\Lambda^2}} & 1\\
\end{bmatrix},
T^{-1}=\frac{1}{detT}
\begin{bmatrix}
1 & \frac{b_1+b_2}{c_2+\sqrt{\Gamma^2-\Lambda^2}} \\
-\frac{b_1-b_2}{c_2+\sqrt{\Gamma^2-\Lambda^2}} & 1
\end{bmatrix},
\eeq
where $detT=\frac{2  \sqrt{\Gamma^2-\Lambda^2}}{c_2+\sqrt{\Gamma^2-\Lambda^2}}$,
except for the case when  $(b_1-b_2)(b_1+b_2) = 0$ and  $c_2 <0$, where
\beq
\label{eqn:T2}
T=
\begin{bmatrix}
1 & -\frac{b_1+b_2}{2c_2} \\
\frac{b_1-b_2}{2c_2} & 1\\
\end{bmatrix},
T^{-1}= 
\begin{bmatrix}
1 & \frac{b_1+b_2}{2c_2} \\
-\frac{b_1-b_2}{2c_2} & 1
\end{bmatrix}.
\eeq

In both situations, we have
\[
detM=detA^N det(I_2-
\begin{bmatrix}
\lo ^N & 0 \\
0 & \lt^N
\end{bmatrix})
=detA^N det(
\begin{bmatrix}
1-\lo ^N & 0 \\
0 & 1-\lt ^N
\end{bmatrix}).
\]
$detM\neq 0$ if and only if $(\lo)^N\neq 1$ and $(\lt)^N\neq 1$. This is clearly true since $\abs{\lo,\lt} \neq 1$.

\underline{Case 2.} If $|\Gamma|=|\Lambda|,$ then $\lo = \lt =(-1)^{k+1} \frac{\Gamma}{\Lambda}$ and we have two repeated eigenvalues.
Perform Jordan decomposition:
\[
\begin{bmatrix}
c_1+c_2 & b_1+b_2 \\
b_1-b_2 & c_1-c_2
\end{bmatrix}
=
\mathcal{T}
\begin{bmatrix}
c_1 & 1 \\
0 & c_1
\end{bmatrix}
\mathcal{T}^{-1},
\]

and
\begin{eqnarray}
\label{eqn:mathcalT}
\mathcal{T}=\begin{bmatrix}
c_2 & 1 \\
b_1-b_2 & 0
\end{bmatrix}, 
&\qquad& \textrm{if} \, b_1\ne b_2, 
\\
\mathcal{T}=\begin{bmatrix}
2b_1 & 0 \\
0 & 1 
\end{bmatrix},
&\qquad& \textrm{if} \, b_1=b_2. \notag
\end{eqnarray}

We define
\[
J=
\begin{bmatrix}
c_1 & 1 \\
0 & c_1
\end{bmatrix}, \qquad
\mathcal{J}
=\frac{(-1)^{k+1}}{\Lambda}
\begin{bmatrix}
c_1 & 1 \\
0 & c_1
\end{bmatrix}
=
\begin{bmatrix}
\lambda_1 &  \frac{(-1)^{k+1}}{\Lambda} \\
0 & \lambda_1
\end{bmatrix},
\]

then
\begin{eqnarray*}
Q^j & = & 
\mathcal{T}
\mathcal{J}^j
\mathcal{T}^{-1}, \quad
\mathcal{J}^j
=
\begin{bmatrix}
\lambda_1^j & \kappa_j \\
0 & \lambda_1^j
\end{bmatrix},
\\
I_2 - Q^N & = &
\mathcal{T}
\begin{bmatrix}
1-(\lambda_1)^N& -\kappa_N \\
0 & 1-(\lambda_1)^N
\end{bmatrix}
\mathcal{T}^{-1},
\end{eqnarray*}
where $\kappa_j=\frac{(-1)^{(k+1)j}}{\Lambda^j}j \Gamma^{j-1}.$ 

In both situations,
$detM\neq 0$ if and only if $(\lo)^N\neq 1$, meaning that we require $N$ to be odd and further, if $k$ is odd, we require  $\Gamma=-\Lambda;$ if $k$ is even, we require $\Gamma=\Lambda.$ In both cases, $\lo=\lt=-1.$

\underline{Case 3.} If $|\Gamma|<|\Lambda|,$ then $\lambda_{1,2}$ are complex, $|\lambda_{1,2} | =1$, $\lo = \overline{\lt}$, still $Q$ is diagonalizable, and similar to  Case 1, $detM\neq 0$ turns to $(\lo)^N\neq 1$ and $(\lt)^N=\overline{(\lo)^N}\neq 1$, i.e. we require
$$
(-1)^{(k+1)N}\left(\frac{\Gamma}{\Lambda}+\sqrt{\left (\frac{\Gamma}{\Lambda} \right)^2-1}\right)^N \ne 1.
$$

\subsection{ Proof of Lemmas \ref{lem:globalpcase1}-\ref{lem:globalpcase3} }
\label{sec:a03}

\subsubsection{Proof of Lemma \ref{lem:globalpcase1}}
\label{sec:a031}
In the proof,  we still use the difference operator $Wu=\pst u-\ppm u=\sum_{l=0}^k \alpha_{j,l}P_{j,l}(x)=\sum_{l=0}^k \alpha_{j,l}P_l(\xi),$ with $
\alpha_{j,l}=0, \quad l=0,\cdots,k-2,\quad j=1,\cdots,N,
$
and

\begin{equation}
\label{eqn:matform2}
M\begin{bmatrix}
\alpha_{1,k-1} \\
\alpha_{1,k} \\
\cdots \\
\alpha_{N-1,k-1} \\
\alpha_{N-1,k} \\
\alpha_{N,k-1} \\
\alpha_{N,k} \\
\end{bmatrix}
=
\begin{bmatrix}
\tau_1  \\
\iota_1 \\
\cdots\\
 \tau_{N-1} \\
\iota_{N-1}\\
 \tau_{N} \\
\iota_{N}
\end{bmatrix},
\end{equation}
where %
\[
\begin{bmatrix}
\tau_j\\
 \iota_j
\end{bmatrix}
=
\begin{bmatrix}
\ot - \ao & -\bt \\
\bo & \ot - \ao
\end{bmatrix}
\begin{bmatrix}
\eta_j 	\\
\theta_j	\\
\end{bmatrix}, \quad
\begin{bmatrix}
\eta_j 	\\
\theta_j	\\
\end{bmatrix}
=\begin{bmatrix}
u-(\ppm u)^+ \\
u_x-(\ppm u)_x^-
\end{bmatrix}_{\jp}.
\]

We will now analyze the inverse of the matrix $M.$ It is known that the inverse of a
nonsingular circulant matrix is also circulant, so is a block-circulant matrix. In particular, %
 \[
 M^{-1}=circ(r_0 , r_1, \cdots, r_{N-1}) \otimes A^{-1} \]
 where  $\otimes$ means Kronecker product for block matrices and $r_j$ is a $2 \times 2$ matrix defined as,
\begin{equation}
\label{eqn:rj}
\begin{aligned}
r_j=&Q^j (I_2-Q^N)^{-1}, \quad j=0,\cdots,N-1\\
    =& TD^j(I_2-D^N)^{-1}T^{-1},
\end{aligned}
\end{equation}
\[
D^j(I_2-D^N)^{-1}=
\begin{bmatrix}
\lambdao & 0 \\
0 & \lambdat
\end{bmatrix}
:=
\begin{bmatrix}
d_1^j & 0 \\
0 & d_2^j
\end{bmatrix}, \quad \textrm{and} \quad \dtj = -d_1^{N-j}.
\]

For the convenience of further analysis, we separate $r_j$ in terms of $\doj$ and $\dtj$,
\begin{eqnarray}
\label{eqn:rj2}
r_j &= &\doj T
\begin{bmatrix}
1 & 0 \\
0 & 0
\end{bmatrix} T^{-1}
+ \dtj T
\begin{bmatrix}
0 & 0 \\
0 & 1
\end{bmatrix} T^{-1} \notag \\
& :=& \doj Q_1 + \dtj (I_2-Q_1), 
\end{eqnarray}
where
 \begin{eqnarray}
\label{eqn:q1}
Q_1 & =&\frac{1}{2\sqrt{\Gamma^2-\Lambda^2}}
\begin{bmatrix}
c_2 + \sqrt{\Gamma^2 - \Lambda^2} & b_1+b_2 \\
b_1 - b_2 & -c_2 + \sqrt{\Gamma^2 - \Lambda^2}
\end{bmatrix},
\end{eqnarray}
when $T$ is given by \eqref{eqn:T}, and

  \begin{eqnarray}
\label{eqn:q12}
Q_1 & =& \frac{1}{2 c_2}
\begin{bmatrix}
2c_2 & b_1+b_2 \\
b_1 - b_2 & 0
\end{bmatrix},
\end{eqnarray}
when $T$ is given by \eqref{eqn:T2}.

For Case 1, eigenvalues $\lambda_{1,2}$ are real. $\sum_{j=0}^{N-1} |d_{1,2}^j | =  \frac{1}{1-|\lambda_{1,2}|}   \frac{1-|\lambda_{1,2}|^{N}}{|1-\lambda_{1,2}^N|}$. Without loss of generality, we assume $|\lo| <1 < |\lt| $, which is equivalent to $\Gamma <0$, then
\begin{eqnarray}
\sum_{j=0}^{N-1} |d_1^j | & \leq & \frac{1}{1-|\lo|} = \frac{|\lt|}{|\lt| -1}, \\
\sum_{j=0}^{N-1} |d_2^j | & \leq & \frac{1}{|\lt|-1}. 
\end{eqnarray}

We let
\begin{eqnarray}
\label{eqn:thetav}
 \begin{bmatrix}
\Xi_j \\
\Theta_j
\end{bmatrix}&:=&
A^{-1}\begin{bmatrix}
\tau_j \\
\iota_j
\end{bmatrix}
= \eta_j V_1 + \theta_j V_2, \qquad j= 1,\cdots, N
\end{eqnarray}
where
\begin{eqnarray}
\label{eqn:v12}
&V_1=
\frac{1}{\Lambda}
\begin{bmatrix}
 -\bo+\frac{k(k+1)}{h}((\ot-\ao)^2+\bo\bt)  \\
 \bo-\frac{k(k-1)}{h}((\ot-\ao)^2+\bo\bt)
\end{bmatrix}, &
V_2=\frac{1}{\Lambda}
\begin{bmatrix}
\ao^2+\bo\bt-\frac{1}{4} \\
-(\ao^2+\bo\bt-\frac{1}{4})
\end{bmatrix}
=
\begin{bmatrix}
-\frac{h}{2k} \\
\frac{h}{2k}
\end{bmatrix},\\
\label{eqn:v12est}
&  \|V_1\|_\infty \leq C \left( 1+ \frac{\max (|\bo|, |\ot-\ao|/h)}{|\la|} \right ), &  \|V_2\|_\infty \leq Ch
\end{eqnarray}
from basic algebraic calculations. Therefore,
\begin{eqnarray}
\label{eqn:alpha1}
\begin{bmatrix}
\alpha_{m,k-1} \\
\alpha_{m,k}
\end{bmatrix}
 &=&\sum_{j=0}^{N-1} r_{j} 
\begin{bmatrix}
\Xi_{j+m} \\
\Theta_{j+m}
\end{bmatrix} ,%
\quad m= 1,\cdots, N,  
\end{eqnarray}
where by periodicity, when $j+m>N$, $\Xi_{j+m} = \Xi_{j+m-N}, \Theta_{j+m} = \Theta_{j+m-N}$.

\bigskip

In summary, we obtain the estimation when $|\lo| <1<|\lt|,$
\begin{eqnarray}
\label{eqn:abd}
\left \|\begin{bmatrix}
\alpha_{m,k-1} \\
\alpha_{m,k}
\end{bmatrix} \right\|_\infty
&\le& \sum_{j=0}^{N-1}  ( | \doj | + | \dtj | ) \left ( \max_j \left | \eta_j  \right | \| Q_1V_1\|_\infty + \max_j |\theta_j| \| Q_1V_2\|_\infty \right )  \notag \\
& + & \sum_{j=0}^{N-1}  | \dtj | \left ( \max_j |\eta_j| \| V_1\|_\infty + \max_j |\theta_j| \| V_2\|_\infty \right ), \notag \\
&\le& C h^{k+1}|u|_{W^{k+1, \infty}(I)} \Big ( \frac{|\lt|+1}{|\lt|-1} \left ( \|Q_1 V_1\|_\infty + h^{-1} \|Q_1 V_2\|_\infty \right ) \notag \\
&& +  \frac{1}{|\lt|-1} \left ( \|V_1\|_\infty + h^{-1} \|V_2\|_\infty \right ) \Big ), \qquad m= 1,\cdots, N.
\end{eqnarray}

Thus, the estimates for the difference between $\pst$ and $P_h^1$ are
\begin{eqnarray}
\label{eqn:estimate12}
\|\pst u -P_h^1 u\|_{L^p(I)} &\le& C h^{k+1}|u|_{W^{k+1, \infty}(I)} \Big ( \frac{|\lt|+1}{|\lt|-1} \left ( \|Q_1 V_1\|_\infty + h^{-1} \|Q_1 V_2\|_\infty \right ) \notag \\
&& +  \frac{1}{|\lt|-1} \left ( \|V_1\|_\infty + h^{-1} \|V_2\|_\infty \right ) \Big ).
\end{eqnarray}
Similar estimates can be proved when $\Gamma>0$ and $|\lo|>1>|\lt|,$
\begin{eqnarray}
\label{eqn:estimate11}
\|\pst u -P_h^1 u\|_{L^p(I)} &\le& C h^{k+1}|u|_{W^{k+1, \infty}(I)} \Big ( \frac{|\lo|+1}{|\lo|-1} \left ( \|(I_2-Q_1) V_1\|_\infty + h^{-1} \|(I_2-Q_1) V_2\|_\infty \right ) \notag \\
&& +  \frac{1}{|\lo|-1} \left ( \|V_1\|_\infty + h^{-1} \|V_2\|_\infty \right ) \Big ),
\end{eqnarray}
and \eqref{eqn:estimate1} is obtained.

\subsubsection{Proof of Lemma \ref{lem:globalpcase2}}

Since $\pst$ is well defined, we know that $\lambda_1^N =-1.$ Therefore, 
we can obtain %
\[
I_2 - Q^N  =
\mathcal{T}
\begin{bmatrix}
2 & \frac{N}{\Gamma} \\
0 & 2
\end{bmatrix}
\mathcal{T}^{-1}, \ 
(I_2-Q^N)^{-1}=
\mathcal{T}
\begin{bmatrix}
\ot &  -\frac{N}{4\Gamma} \\
0 & \ot
\end{bmatrix}
\mathcal{T}^{-1},
\]
\beq
\label{eqn:rjcase2}
r_j = Q^j(I_2-Q^N)^{-1}
=\frac{(-1)^j}{2}I_2 + (-1)^j \frac{-N+2j}{4\g} Q_2
,\eeq
where 
\beq
\label{eqn:q22}
Q_2 = \mathcal{T}
\begin{bmatrix}
0 & 1 \\
0 & 0
\end{bmatrix},
\mathcal{T}^{-1}
= 
\begin{bmatrix}
c_2 & b_1 + b_2 \\
b_1 - b_2 & -c_2
\end{bmatrix}.
\eeq

Therefore, we have for $m = 0, \cdots, N-1$:
\begin{eqnarray*}
\left \|\begin{bmatrix}
\alpha_{m,k-1} \\
\alpha_{m,k}
\end{bmatrix} \right\|_\infty
&\le&
C N \left (1 +     \frac{N}{ |\Gamma|} \|Q_2\|_\infty   \right )  \max_j\left \| \begin{bmatrix} 
\Xi_{j} \\
\Theta_{j}
\end{bmatrix}  \right\|_\infty , \\
& \le & C  h^{k} |u|_{W^{k+1,\infty}(I)} \left (1 +     \frac{h^{-1} \|Q_2\|_\infty}{|\Gamma|}   \right ) \left ( \|V_1\|_\infty + h^{-1} \|V_2\|_\infty \right ).
\end{eqnarray*}
Similar to  Lemma \ref{lem:globalpcase1}, we can estimate $\|\pst u -P_h^1 u\|_{L^p(I)} $ and \eqref{eqn:estimate2} follows.

\subsubsection{Proof of Lemma \ref{lem:globalpcase3}}
In Case 3,  $\lambda_{1,2}$ are conjugate to each other and $|\lambda_{1,2}| = 1$. Therefore, $ \delta' \geq 0,$ and 
\begin{align*}
\sum_{j=0}^{N-1} |\doj |  & =\sum_{j=0}^{N-1} \left|\lambdao\right| =  \frac{N}{|1-\lambda_1^N|}=\sum_{j=0}^{N-1} |\dtj|.
\end{align*}
Similar to \eqref{eqn:abd}, we obtain
\begin{eqnarray*}
\left \|\begin{bmatrix}
\alpha_{m,k-1} \\
\alpha_{m,k}
\end{bmatrix} \right\|_\infty
&\le& \sum_{j=0}^{N-1}  ( | \doj | + | \dtj | ) \left ( \max_j |\eta_j| \| Q_1V_1\|_\infty + \max_j |\theta_j| \| Q_1V_2\|_\infty \right )  \notag \\
& + & \sum_{j=0}^{N-1}  | \dtj | \left ( \max_j |\eta_j| \| V_1\|_\infty + \max_j |\theta_j| \| V_2\|_\infty \right ), \notag \\
& \le & C  h^{k+1} |u|_{W^{k+1, \infty}(I)}  h^{-(\delta'+1)}\left ( \|Q_1V_1\|_\infty + h^{-1} \|Q_1V_2\|_\infty +  \|V_1\|_\infty + h^{-1} \|V_2\|_\infty \right )   \notag \\
& \le & C  h^{k+1} |u|_{W^{k+1, \infty}(I)}   h^{-(\delta'+1)} \|Q_1\|_\infty  \left ( \|V_1\|_\infty + h^{-1} \|V_2\|_\infty \right )   \notag
\end{eqnarray*} 
and we reach the estimation \eqref{eqn:estimate3}.

\subsection{ Proof of Lemma \ref{lem:globalpcase2optimal} }
\label{sec:a04}
From \eqref{eqn:thetav} and \eqref{eqn:alpha1},  we have%
\begin{eqnarray}
\label{eqn:alpha}
\begin{bmatrix}
\alpha_{m,k-1} \\
\alpha_{m,k}
\end{bmatrix}
 &=&\sum_{j=0}^{N-1} r_{j} 
\begin{bmatrix}
\Xi_{j+m} \\
\Theta_{j+m}
\end{bmatrix} ,%
\quad m= 1,\cdots, N, \notag \\
 &=:& {U}_1V_1 + {U}_2V_2,  \notag
\end{eqnarray}
where $ {U}_1= \sum_{j=0}^{N-1} r_j \eta_{j+m},  {U}_2= \sum_{j=0}^{N-1} r_j \theta_{j+m} $. We first estimate $ {U}_1,$ then $ {U}_2$ can be estimated in a similar way. 
From \eqref{eqn:rjcase2}, %
\begin{eqnarray}
\label{eqn:tut1}
 {U}_1  & =  &
\ot I_2 \sum_{j=0}^{N-1} (-1)^j \eta_{j+m} + 
 \frac{Q_2}{2\Gamma}
\sum _{j=0}^{N-1} (-1)^j \frac{-N+2j}{2}\eta_{j+m}.
\end{eqnarray}

By Lemma \ref{lem:rep}, the first term in \eqref{eqn:tut1} can be estimated by
\begin{eqnarray}
\label{eqn:s21}
|\sum_{j=0}^{N-1} (-1)^j \eta_{j+m}| & = & |\sum_{j'=0}^{\frac{N-1}{2}} (\eta_{2j'+m} - \eta_{2j'+1+m}) +  \eta_{N-1+m}| \notag \\
& \leq & \frac{N-1}{2}C h^{k+2} | u|  _{W^{k+2,\infty}(I)} +  C h^{k+1} | u|  _{W^{k+1,\infty}(I)}  \notag \\
&\leq & C h^{k+1}\| u\|  _{W^{k+2,\infty}(I)},
\end{eqnarray}
because $N$ must be odd from Lemma  \ref{lem:globalp}.
The second term in \eqref{eqn:tut1} can be estimated by using \eqref{lem:rep1},
\begin{eqnarray}
\label{eqn:s22}
&& \left|\sum _{j=0}^{N-1}(-1)^j \frac{-N+2j}{2} \eta_{j+m} - \mu  h^{k+1}S_{k+1} - \mu_2  h^{k+2} S_{k+2} \right| \notag \\
 &&\le C  \sum_{j=0}^{N-1} \left |\frac{-N +2j}{2}\right|   h^{k+3}  | u|  _{W^{k+3,\infty}(I)} \le C h^{k+1}  | u|  _{W^{k+3,\infty}(I)},
\end{eqnarray}
where $S_{k+1}, S_{k+2}$ are defined as:
\begin{eqnarray*}
S_{k+\nu} & : = & \sum_{j=0}^{N-1} (-1)^j \frac{-N+2j}{2} u^{(k+\nu)}(x_{j+m+\ot}), \quad \nu = 1,2. \\
 \end{eqnarray*} 

We assume $u \in W^{k+4,1}{(I)}$. Then $u^{(k+1)} \in W^{3,1}(I)$, is periodic, and has   the following Fourier series expansion $u^{(k+1)}(x) = \sum_{n=-\infty}^{\infty} \hat f(n) e^{2\pi inx/L}, L= b-a$, where its fourier coefficient $\hat f(n)$ satisfies:
\beq
\label{eqn:fregular}
\left | \hat f(n) \right | \leq C\frac{|u|_{W^{k+4,1}{(I)}}}{1+|n|^3}.
\eeq
 Since $x_{j+\ot} = j\Delta x = j \frac{L}{N}, \ j=0,\cdots, N-1$, then $u^{(k+1)}(x_{j+\ot}) = \sum_{n=-\infty}^{\infty} \hat f(n) \omega^{jn}$ with $\omega = e^{i\frac{2\pi}{N}}.$
Then 
\[S_{k+1} = \sum_{j=0}^{N-1}   (-1)^j \frac{-N+2j}{2} \sum_{n=-\infty}^{\infty} \hat f(n) \omega^{(j+m)n}. \]

Due to \eqref{eqn:fregular}, $\sum |\hat f(n)| $ is convergent and we can switch the order of summation, which results in
\beq
\label{eqn:s2def}
S_{k+1} = \sumn \hat f(n) W(n), \quad \textrm{where} \quad W(n) = \frac{-2 \omega^{(m+1)n}}{(1+\omega^n)^2}.
\eeq

Since $N$ is odd,  $\omega ^n = e^{2\pi i \frac{n}{N}} \neq -1, \forall n.$ Hence, $W(n)$ and $S_{k+1}$ are well defined. 
Because $W(n)$ is $N$-periodic, it's helpful to split $S_{k+1}$ into blocks of size $N$ as
$$
S_{k+1} = \sum_{l=-\infty}^{\infty} S_{k+1}^l,   \quad \textrm{where} \quad  S_{k+1}^l=\sum_{n=lN-\frac{N-1}{2}}^{lN+\frac{N-1}{2}}\hat f(n) W(n).
$$

Let's estimate $S_{k+1}^0$ first. For $|n| \leq [\frac{3N}{8}]$,   $|W(n)| = \frac{2}{|1+\omega^n|^2}  \le \frac{2}{|1+e^{i3\pi/4}|^2}=\frac{2}{2-\sqrt{2}}$. For  other $n$, 
 $|W(n)| \le |W(\frac{N-1}{2})|= \frac{2}{|1+\omega^{(N-1)/2}|^2}  \le CN^2$ from Taylor expansions.
\begin{eqnarray*}
|S_{k+1}^0| &\le & \sum_{n=-[\frac{3N}{8}]}^{[\frac{3N}{8}]} \left |  \hat f(n) W(n)  \right | +\sum^{n= -[\frac{3N}{8} ]-1}_{-\frac{N-1}{2}} \left |  \hat f(n) W(n)  \right | +\sum_ {n= [\frac{3N}{8} ]+1}^{\frac{N-1}{2}} \left |  \hat f(n) W(n)  \right |\\
& \leq & \frac{2}{2-\sqrt{2}}  \sum_{n=-[\frac{3N}{8}]}^{[\frac{3N}{8}]} \left |  \hat f(n)   \right | +CN^2 \sum^{n= -[\frac{3N}{8} ]-1}_{-\frac{N-1}{2}} \left |  \hat f(n)    \right | +CN^2\sum_ {n= [\frac{3N}{8} ]+1}^{\frac{N-1}{2}} \left |  \hat f(n)   \right |  \\
& \leq & \frac{2}{2-\sqrt{2}}  \sum_{n=-[\frac{3N}{8}]}^{[\frac{3N}{8}]} \left |  \hat f(n)   \right | +CN^2  \frac{1}{1+(\frac{3N}{8})^3}  (\frac{N}{4}+2) |u|_{W^{k+4,1}{(I)}}\\
& \leq & C \left (\sum_{n=-\frac{N-1}{2}}^{\frac{N-1}{2}}\frac{1}{1+|n|^3} +   \frac{1}{ (\frac{3}{8})^3} \right ) |u|_{W^{k+4,1}{(I)}}.
\end{eqnarray*}
Then, in a similar way,
\begin{eqnarray*}
|S_{k+1}^l|   \le C \left (\sum_{n=lN-\frac{N-1}{2}}^{lN+\frac{N-1}{2}}\frac{1}{1+|n|^3} +   \frac{1}{ (|l|+\frac{3 }{8})^3} \right ) |u|_{W^{k+4,1}{(I)}}.
\end{eqnarray*}
Therefore,
\begin{eqnarray}
\label{eqn:skp1}
|S_{k+1}|   \le C \left (\sum_{n=-\infty}^{\infty}\frac{1}{1+|n|^3} +  \sum_{l=-\infty}^{\infty} \frac{1}{ (|l|+\frac{3 }{8})^3} \right ) |u|_{W^{k+4,1}{(I)}} \le C |u|_{W^{k+4,1}{(I)}}.
\end{eqnarray}
By similar Fourier expansion technique, we can show
\begin{eqnarray}
\label{eqn:skp2}
|S_{k+2}| \leq C N  |u|_{W^{k+4,1}{(I)}}= C h^{-1}  |u|_{W^{k+4,1}{(I)}}.
\end{eqnarray}

Combine \eqref{eqn:skp1}, \eqref{eqn:skp2} with \eqref{eqn:tut1}, \eqref{eqn:s21} and \eqref{eqn:s22}, we get
\beq
\norm{  {U}_1 }_\infty  \leq Ch^{k+1} \|u\|_{W^{k+4,\infty}{(I)}} \left ( 1 + \frac{\norm{Q_2}_\infty}{|\Gamma|} \right ).
\eeq
Similarly, by \eqref{eqn:thetaest} and the Fourier expansion technique
 \beq
\norm{ {U}_2 }_\infty  \leq Ch^{k} \|u\|_{W^{k+4,\infty}{(I)}}  \left ( 1 + \frac{\norm{Q_2}_\infty}{|\Gamma|} \right ).
\eeq
Therefore, 
\begin{eqnarray}
\label{eqn:alpha2}
\left \|\begin{bmatrix}
\alpha_{m,k-1} \\
\alpha_{m,k}
\end{bmatrix} \right\|_\infty
&\le&
\norm{ {U}_1 }_\infty \norm{V_1}_\infty + \norm{ {U}_2}_\infty \norm{V_2}_\infty, \notag \\
&\le & C h^{k+1} \|u\|_{W^{k+4,\infty}(I)} \left (1+\frac{ \norm{Q_2}_\infty}{|\Gamma|} \right ) (\|V_1\|_\infty + h^{-1} \|V_2 \|_\infty), \quad m = 1,\cdots, N, \notag  
\end{eqnarray}
  and \eqref{eqn:estimate2optimal} is obtained.

\subsection{ Proof of Lemma \ref{lem:globalpcase3optimal} }
\label{sec:a05}

From the discussion in Lemma \ref{lem:globalp}, we can write $\lambda_{1,2} = e^{\pm i\theta}$ and assume $ \theta \in (0,\pi)$. First, we want to make clear of the conditions on $\delta, \delta'.$ Since $|\lo|=|\lo^N|=1,$ we have $\delta, \delta' \geq 0.$
Because
$1-\lo^N = (\omega^n)^N - (e^{i\theta})^N = (\omega^n - e^{ i\theta})(\sum_{l=0}^{N-1} (\omega^n)^{N-1-l}(e^{ i\theta})^l)$, thus $\abs{1-\lo^N} \leq N \abs{\omega^n - e^{ i\theta}}, \forall n.$ With the assumption $\abs{1-\lo^N } \sim Ch^{\delta'}$, we get $\abs{\omega^n - e^{ i\theta} } \geq Ch^{\delta'+1}.$ Particularly, when $n=0$, we have $\abs{1-\lo} \geq Ch^{\delta'+1}$, hence $\delta/2 \leq \delta'+1.$ 

 Similar to \eqref{eqn:alpha1} in Case 1, we can get
\begin{eqnarray}
\label{eqn:alpha}
\begin{bmatrix}
\alpha_{m,k-1} \\
\alpha_{m,k}
\end{bmatrix}
 &=&\sum_{j=0}^{N-1} r_{j} 
\begin{bmatrix}
\Xi_{j+m} \\
\Theta_{j+m}
\end{bmatrix} ,%
\quad m= 1,\cdots, N, \notag \\
 &=:& \mathcal{U}_1V_1 +\mathcal{U}_2V_2,  \notag
\end{eqnarray}
where
\begin{eqnarray*}
\mathcal{U}_1 =Q_1 \sum_{j=0}^{N-1} \eta_{j+m} \doj  +(I_2-Q_1) \sum_{j=0}^{N-1} \eta_{j+m} \dtj,	&	\doj = \frac{e^{ij\theta}}{1-e^{iN\theta}}, \\
\mathcal{U}_2 =Q_1 \sum_{j=0}^{N-1} \theta_{j+m} \doj  +(I_2-Q_1) \sum_{j=0}^{N-1} \theta_{j+m} \dtj,	&	\dtj =  -d_1^{N-j}=\frac{-e^{i(N-j)\theta}}{1-e^{iN\theta}}.
\end{eqnarray*}
We introduce:
\begin{eqnarray*}
\s_1 &= &  \frac{1}{1-e^{iN\theta}} \sum_{j=0}^{N-1} e^{ij\theta} u^{(k+1)}(x_{j+m+\ot}) ,\\
\s_2 &= & \frac{-1}{1-e^{iN\theta}} \sum_{j=0}^{N-1} e^{i(N-j)\theta} u^{(k+1)}( x_{j+m-\ot}).
\end{eqnarray*}

Then by Lemma \ref{lem:rep}:
\begin{align*}
&\abs{\mathcal{U}_1-\mu h^{k+1} Q_1 \s_1-\mu h^{k+1} (I_2-Q_1) \s_2 } \leq Ch^{k+1} (1+\|Q_1\|_\infty) |u|_{W^{k+2, \infty}(I)},\\
&\abs{\mathcal{U}_2-\rho h^{k} Q_1 \s_1-\rho h^{k} (I_2-Q_1) \s_2 } \leq Ch^{k} (1+\|Q_1\|_\infty) |u|_{W^{k+2, \infty}(I)}.
\end{align*}

Therefore,
\beq
\label{eqn:mathcalu}
\abs{\mathcal{U}_\nu } \leq Ch^{k+2-\nu}\left ( (1+ \|{Q_1}\|_\infty)|u|_{W^{k+2, \infty}} + \|{Q_1}\|_\infty (1+ \max (\abs{\s_1},\abs{\s_2}) ) \right ), \quad \nu = 1,2.
\eeq

By using similar Fourier expansion: $u^{(k+1)}(x_{j+\ot}) = \sumn \hat f(n)\omega^{jn} $. Since now we assume $u \in W^{k+3,\infty}(I)$,  $\abs{\hat f(n)} \leq C\frac{1}{1+|n|^2}|u|_{W^{k+3,1}(I)}.$
\begin{eqnarray*}
\s_1  & = & \frac{1}{1-e^{iN\theta}} \sumn \hat f(n) \sum_{j=0}^{N-1} e^{ij\theta} \omega^{(j+m)n} = \sumn \hat f(n) \w_1(n),\\
\s_2 & = & \frac{-1}{1-e^{iN\theta}} \sumn \hat f(n) \sum_{j=0}^{N-1} e^{i(N-j)\theta}\omega^{(j+m-1)n} =  \sumn \hat f(n) \w_2(n),
\end{eqnarray*}
where from simple algebra
\[
\w_1(n) = \frac{\omega^{mn}}{1-e^{i\theta} \omega ^ n}, \quad \w_2(n) = \frac{\omega^{(m-1)n}}{1-e^{-i\theta} \omega ^ n}.
\]
From the discussion at the beginning of the proof, we have $\abs{\w_2(n)}=|\lo-\omega^n|^{-1}\leq Ch^{-(\delta'+1)},$ and similarly $\abs{\w_1(n)}\leq Ch^{-(\delta'+1)}.$
Since $\s_1$ and $\s_2$ can be estimated in the same way, we only show  details for $\s_2$ in what follows.
Similar to the proof of Lemma \ref{lem:globalpcase2optimal}, we split $\s_2$ into blocks of size $N$,
$$
\s_2 = \sum_{l=-\infty}^{\infty} \s_{2}^l,   \quad \textrm{where} \quad \s_{2}^l=\sum_{n=lN}^{(l+1)N-1}\hat f(n) \w_2(n).
$$

With the assumption that $0 \leq \delta/2 \leq 1$, there $\exists \,n_0 \sim O(h^{\delta/2 -1})$ s.t. $2\pi \frac{n_0}{N} \le \theta < 2\pi \frac{n_0+1}{N}$. Let $n_1 = \floor{n_0/2}, n_2 = 2n_0 - n_1$, %
then for $n_1 \leq n \leq n_2$,  $\abs{\hat f(n)} \leq C\frac{1}{1+ n_1^2}|u|_{W^{k+3,1}(I)}.$
For other $n$, $\abs{\w_2(n)} \leq \abs{\w_2(n_1)} \leq  \frac{1}{2\abs{\sin(\pi n_1/N-\theta/2)}} \leq Ch^{-\delta/2}$. Thus,
\begin{align*}
\abs{\s_2^0} & \leq Ch^{-\delta/2} \left (\sum_{n=0}^{n_1-1} + \sum_{n= n_2+1}^{N-1} \abs{\hat f(n)} \right ) + Ch^{-(\delta'+1)}\sum_{n=n_1}^{n_2} \abs{\hat f(n)} \\
	& \leq  C \left ( h^{-\delta/2}  \sum_{n= 0}^{N-1} \frac{1}{1+|n|^2} + h^{-(\delta'+1)}(n_2-n_1+1)\frac{1}{1+n_1^2} \right ) |u|_{W^{k+3,1}(I)} \\
	& \leq  C \left ( h^{-\delta/2}  \sum_{n= 0}^{N-1} \frac{1}{1+|n|^2} + h^{-(\delta'+1)}h^{\delta/2-1} h^{2-\delta} \right ) |u|_{W^{k+3,1}(I)} \\
	& \leq C \left (h^{-\delta/2} \sum_{n= 0}^{N-1} \frac{1}{1+|n|^2} +h^{-\delta'-\delta/2} \right ) |u|_{W^{k+3,1}(I)}.
\end{align*}
Using similar approaches, for $l\neq 0$,
$$
\abs{\s_2^l} \leq C \left (h^{-\delta/2}  \sum_{n= lN}^{(l+1)N-1} \frac{1}{1+|n|^2}+ h^{-\delta'+\delta/2} \frac{1}{|n_1/N+l|^2}  \right ) |u|_{W^{k+3,1}(I)}.
$$
Summing up, we reach the estimation
\begin{align}
\abs{\s_2} & \leq C \left (h^{-\delta/2}  \sum_{n=-\infty}^{\infty} \frac{1}{1+|n|^2}+ h^{-\delta'-\delta/2} + h^{-\delta'+\delta/2} \sum_{l\in \mathbb{N}, l\neq 0} \frac{1}{|l|^2}   \right ) |u|_{W^{k+3,1}(I)} \notag \\
& \leq C h^{-\delta'-\delta/2} |u|_{W^{k+3,1}(I)}.
\label{eqn:s2singular}
\end{align}
Similarly, we obtain
\beq
\label{eqn:s1singular}
\abs{\s_1} \leq   C h^{-\delta'-\delta/2}  |u|_{W^{k+3,1}(I)}.
\eeq	

Combine \eqref{eqn:s2singular}, \eqref{eqn:s1singular} and \eqref{eqn:mathcalu}, we get
$$
|\mathcal{U}_\nu| \leq 
Ch^{k+2-\nu}(1+h^{-(\delta'+\delta/2)}\|Q_1\|_\infty) \|u\|_{W^{k+3,\infty}(I)}, \quad \nu = 1,2,
$$
and   \eqref{eqn:estimate3optimal} follows.

\subsection{ Detailed discussions on the choice of the   $T$ matrix as in   \eqref{eqn:T} or \eqref{eqn:T2} }
\label{sec:a1}
We discuss what parameters result in $| b_1 \pm b_2|  = 0,$ under the assumption that $\ao$ has no dependence on $h$, $\bo=\tilde{\bo} h^{A_1}, \bt=\tilde{\bt} h^{A_2},\tilde{\bo}, \tilde{\bt}$ are nonzero constants that do not depend on $h.$
\begin{eqnarray*}
b_1 - b_2 & = & (-\bo + \frac{k(k-1)}{2h})(1-\bt\frac{2k(k-1)}{h}) + \frac{k(k-1)}{h} 2\ao^2 \\
& = & (-\tilde \bo h^{A_1}+ \frac{k(k-1)}{2}{h^{-1}})(1-{2k(k-1)}\tilde \bt h^{A_2-1}) + {k(k-1)} 2\ao^2h^{-1},
\end{eqnarray*}
\begin{eqnarray*}
b_1 + b_2 & = & (-\bo + \frac{k(k+1)}{2h})(1-\bt\frac{2k(k+1)}{h}) + \frac{k(k+1)}{h} 2\ao^2 \\
& = & (-\tilde \bo h^{A_1}+ \frac{k(k+1)}{2}{h^{-1}})(1-{2k(k+1)}\tilde \bt h^{A_2-1}) + {k(k+1)} 2\ao^2h^{-1}.
\end{eqnarray*}

If $b_1 - b_2 = 0, \forall h < h_0$, then
\begin{itemize}
\item  $\ao \neq 0$, then $A_1 = -1, \ A_2 = 1$ and $\tilde \bo, \tilde \bt$ satisfies
\beq
\label{eqn:b1mb2}
(-\tilde \bo + \frac{k(k-1)}{2})(1-{2k(k-1)}\tilde \bt ) + {k(k-1)}2\ao^2 = 0.
\eeq
\end{itemize}

Similarly, for $b_1 + b_2 =0, \forall h < h_0$, then
 \begin{itemize}
\item  $\ao \neq 0$, $A_1 = -1, \ A_2=1$ and $\tilde \bo, \tilde \bt$ satisfies
\beq
\label{eqn:b1pb2}
(-\tilde \bo + \frac{k(k+1)}{2})(1-{2k(k+1)}\tilde \bt ) + {k(k+1)} 2\ao^2 = 0.
\eeq
\end{itemize}

\subsection{ Detailed discussions on Case 2 }
\label{sec:a2}

Parameter choices for $| \Gamma|  = | \Lambda | $ imply
\begin{eqnarray*}
\Gamma \pm \Lambda & = & \bo + \frac{k^2(k^2-1)}{h^2}\bt + \frac{k(k\pm1)}{h}(-2\ao^2 -2\bo\bt) + \frac{-k^2 \pm k}{2h} \\
& = & (\bo-\frac{k(k\mp1)}{2h})(1-2\bt \frac{k(k\pm1)}{h})-\frac{k(k \pm 1)}{h}2\ao^2=0,
\end{eqnarray*}
which indicates 
\begin{itemize}
\item if $\ao \neq 0$, then $b_1 \pm b_2$ can be greatly simplified as follows.
\begin{itemize}
	\item If $\Gamma + \Lambda = 0,$ then $k$ is odd from Lemma \ref{lem:globalp}, and
	\begin{align*}
	& b_1+b_2 = \frac{k}{h}  \left (1- \bt \frac{2k(k+1)}{h} \right), 	\\
	& b_1-b_2 = -\frac{2}{k+1}  \left( \bo - \frac{k(k-1)}{2h} \right),	\\
	& \la = -\frac{1}{k+1} \left( \bo -\frac{k^2}{h} + \frac{k^2(k^2-1)}{h^2}\bt \right ).
	\end{align*}
	\item If $\Gamma -\Lambda = 0,$ then $k$ is even from Lemma \ref{lem:globalp}, and
	\begin{align*}
	& b_1+b_2 = \frac{2}{k-1} \left (\bo - \frac{k(k+1)}{2h} \right), 	\\
	& b_1-b_2 = -\frac{k}{h}  \left(1- \bt \frac{2k(k-1)}{h} \right),		\\
	& \la = -\frac{1}{k-1} \left( \bo -\frac{k^2}{h} + \frac{k^2(k^2-1)}{h^2}\bt \right ), \quad k>1.\\
	\end{align*}
\end{itemize}

\item If $\ao = 0,$ then 
\beq
\label{eqn:case22}
\bo = \frac{k(k\pm 1)}{2h}, \mathrm{or} \ \bt = \frac{h}{2k(k \pm 1)}.
\eeq
\end{itemize}

\bibliographystyle{abbrv}
\bibliography{ac@msu} 

\end{document}